\documentclass{amsart}

\usepackage{amssymb,amscd}

\newdimen\theight
\def\TeXref#1{%
              \leavevmode\vadjust{\setbox0=\hbox{{\tt
                      \quad\quad  {\small \textrm #1}}}%
              \theight=\ht0
              \advance\theight by \lineskip
              \kern -\theight \vbox to
              \theight{\rightline{\rlap{\box0}}%
              \vss}%
              }}%

\theoremstyle{plain}

\newtheorem{thm}{Theorem}[section]
\newtheorem{lem}[thm]{Lemma}
\newtheorem{cor}[thm]{Corollary}
\newtheorem{prop}[thm]{Proposition}

\theoremstyle{definition}

\newtheorem{defn}[thm]{Definition}

\newtheorem{ex}[thm]{Example}

\newtheorem{quest}[thm]{Question}

\theoremstyle{remark}

\newtheorem{rem}{Remark}

\newtheorem{claim}{Claim}

\newcommand{\BB}{\mathcal{B}}
\newcommand{\EE}{\mathcal{E}}
\newcommand{\FF}{\mathcal{F}}

\newcommand{\KK}{\mathcal{K}}

\newcommand{\CC}{\mathcal{C}}
\newcommand{\MM}{\mathcal{M}}

\newcommand{\HH}{\mathcal{H}}
\newcommand{\VV}{\mathcal{V}}
\newcommand{\UU}{\mathcal{U}}

\newcommand{\NN}{\mathcal{N}}

\newcommand{\XX}{\mathcal{X}}
\newcommand{\YY}{\mathcal{Y}}

\newcommand{\AAA}{\mathcal{A}}

        \newcommand{\field}[1]{\text{$\mathbb{#1}$}}
        \newcommand{\N}{\field{N}}
        \newcommand{\Z}{\field{Z}}
        
        \newcommand{\R}{\field{R}}

\newcommand{\supp}{\operatorname{supp}}
\newcommand{\codim}{\operatorname{codim}}

\newcommand{\id}{\operatorname{id}}

\newcommand{\sign}{\operatorname{sign}}
\newcommand{\Fix}{\operatorname{Fix}}

\newcommand{\pr}{\operatorname{pr}}
\newcommand{\Diff}{\operatorname{Diff}}

\newcommand{\Graph}{\operatorname{Graph}}
\newcommand{\Coin}{\operatorname{Coin}}
\newcommand{\Int}{\operatorname{Int}}

\newcommand{\cinf}{\text{$C^\infty$}}

\newcommand{\D}{\text{$\Delta$}}

\title[Transversality and Lefschetz numbers for foliation maps]{Transversality and Lefschetz numbers for foliation maps}
\author[J.A. \'Alvarez L\'opez]{Jes\'us A. \'Alvarez L\'opez}
\address{Departamento de Xeometr\'{\i}a e Topolox\'{\i}a\\
         Facultade de Matem\'aticas\\
         Universidade de Santiago de Compostela\\
         15782 Santiago de Compostela\\ Spain}
\email{jalvarez@usc.es}
\thanks{The first author is partially supported by the MEC grant MTM2004-08214}

\author[Y.A. Kordyukov]{Yuri A. Kordyukov}
\address{Institute of Mathematics\\ Russian Academy of Sciences\\
112 Chernyshevsky street\\ 450077 Ufa\\ Russia}
\email{yurikor@matem.anrb.ru}
\thanks{The second author is partially supported by the RFBR grant 06-01-00208.}



\begin{document}

\begin{abstract}
Let $\FF$ be a smooth foliation on a closed Riemannian manifold $M$,
and let $\Lambda$ be a transverse invariant measure of $\FF$.
Suppose that $\Lambda$ is absolutely continuous with respect to the
Lebesgue measure on smooth transversals. Then a topological
definition of the $\Lambda$-Lefschetz number of any leaf preserving
diffeomorphism $(M,\FF)\to(M,\FF)$ is given. For this purpose,
standard results about smooth approximation and transversality are
extended to the case of foliation maps. It is asked whether this
topological $\Lambda$-Lefschetz number is equal to the analytic
$\Lambda$-Lefschetz number defined by Heitsch and Lazarov which
would be a version of the Lefschetz trace formula. Heitsch and
Lazarov have shown such a trace formula when the fixed point set is
transverse to $\FF$.
\end{abstract}

\maketitle

\tableofcontents


\section{Introduction}

Let $\FF$ be a smooth foliation of codimension $q$ on a closed
Riemannian manifold $M$, and let $\Lambda$ be a transverse invariant
measure of $\FF$. Continuing the ideas of  Connes' index theory for
foliations \cite{Connes1979,Connes1982}, Heitsch and Lazarov use
$\Lambda$ to define an analytic version of Lefschetz number
for any leaf preserving diffeomorphism $\phi:(M,\FF)\to(M,\FF)$
which is geometric with respect to a leafwise Dirac complex
\cite{HeitschLazarov}. For the sake of simplicity, consider only the
case of the leafwise de~Rham complex. Then, by using the Hodge
isomorphism on the leaves, the ``local super-trace'' of the induced
map on the reduced $L^2$ cohomology of the leaves can be defined as
a density on the leaves. The pairing of this leafwise density with
$\Lambda$ is a measure on $M$, whose integral is the {\em analytic
$\Lambda$-Lefschetz number\/}, which will be denoted by
$L_{\Lambda,\text{\rm an}}(\phi)$.

Let $\Fix(\phi)$ denote the fixed point set of $\phi$. In
\cite{HeitschLazarov}, Heitsch and Lazarov have also proved a Lefschetz
trace formula when $\Fix(\phi)$ is a regular submanifold transverse
to $\FF$. Assume that $\dim\Fix(\phi)=q$ for the sake of simplicity;
thus $\Fix(\phi)$ is a closed transversal of $\FF$. Then the
Lefschetz trace formula is
\begin{equation}\label{e:Lefschetz}
L_{\Lambda,\text{\rm an}}(\phi)=\int_{\Fix(\phi)}\epsilon_\phi\,\Lambda\;,
\end{equation}
where $\epsilon_\phi:\Fix(\phi)\to\{\pm1\}$ is the locally constant function defined by
$$
\epsilon_\phi(x)=\sign\det(\id-\phi_*:T_x\FF\to T_x\FF)\;.
$$
The right hand side of~\eqref{e:Lefschetz} can be called the {\em
topological $\Lambda$-Lefschetz number\/} of $\phi$, and denoted by
$L_{\Lambda,\text{\rm top}}(\phi)$. Thus~\eqref{e:Lefschetz} becomes
\begin{equation}\label{e:Lefschetz'}
L_{\Lambda,\text{\rm an}}(\phi)=L_{\Lambda,\text{\rm top}}(\phi)\;.
\end{equation}

In this paper, we give some steps towards a generalization
of~\eqref{e:Lefschetz'} by relaxing the transversality condition on
$\Fix(\phi)$. The motivation is that existence of a closed
transversal is a strong restriction on $\FF$. But we have to assume
that $\Lambda$ is absolutely continuous with respect to the Lebesgue
measure on smooth transversals; {\em i.e.\/}, it is given by a
continuous density on smooth transversals.

For that purpose, we extend some standard results about maps to the
case of {\em foliation maps\/} between foliated manifolds (those
which map leaves to leaves). This kind of study was begun in
\cite{AlvMasa2006,AlvMasa}, where the set of smooth foliation maps
is endowed with certain topologies, called  the {\em strong\/} and
{\em weak plaquewise topologies\/}. That study is recalled and
pursued further here; in particular, we study certain version of
transversality for foliation maps, called {\em ls-transversality\/}.

It is shown that any leaf preserving diffeomorphism
$\phi:(M,\FF)\to(M,\FF)$ is ``homotopic along the leaves'' to a leaf
preserving map $\xi:(M,\FF)\to(M,\FF)$ whose graph is ls-transverse
to the diagonal of $M\times M$. With this condition, it turns out
that $\Fix(\xi)$ is a closed regular submanifold of $M$ of dimension
$q$. Let $\Fix_0(\xi)\subset\Fix(\xi)$ denote the open subset where
$\Fix(\xi)$ is transverse to $\FF$, which is a transversal of $\FF$.
Then we prove that the {\em topological $\Lambda$-Lefschetz
number\/} of $\phi$ is well defined by the formula
\begin{equation}\label{e:L Lambda,top}
L_{\Lambda,\text{\rm top}}(\phi)=\int_{\Fix_0(\xi)}\epsilon_\xi\,\Lambda\;;
\end{equation}
{\em i.e.\/}, we show that this integral is defined and is independent of $\xi$.

With more generality, the {\em coincidence point set\/} of two
foliation maps $\phi,\psi:(M,\FF)\to(M,\FF)$ is
$$
\Coin(\phi,\psi)=\{x\in M\ |\ \phi(x)=\psi(x)\}\;.
$$
If $\phi$ and $\psi$ are ``transversely smooth'' and induce the same
map on the leaf space $M/\FF$, then, with the obvious generalization
of the notation of~\eqref{e:L Lambda,top}, it is proved that their
{\em topological $\Lambda$-coincidence\/} is well defined by a
formula
$$
\Coin_{\Lambda,\text{\rm top}}(\phi,\psi)=\int_{\Coin_0(\xi,\zeta)}\epsilon_{\xi,\zeta}\,\Lambda\;.
$$
Therefore, for any ``transversely smooth'' leaf preserving map
$\phi:M\to M$, the {\em topological $\Lambda$-Lefschetz number\/} of
$\phi$ is well defined by
\begin{equation}\label{e:L Lambda,top'}
L_{\Lambda,\text{\rm top}}(\phi)=\Coin_{\Lambda,\text{\rm top}}(\id,\phi)\;.
\end{equation}

This work raises the following question.

\begin{quest}
With the definition of $L_{\Lambda,\text{\rm top}}(\phi)$ given
by~\eqref{e:L Lambda,top}, is the Lefschetz trace
formula~\eqref{e:Lefschetz'} valid without the assumption that the
fixed point set of $\phi$ is a transversal?
\end{quest}

Besides its own interest, an affirmative answer to this question
would be helpful to describe the Lefschetz distribution of a Lie
foliation on a closed manifold \cite{AlvKordy}.

There are some other works related with measurable Lefschetz numbers
for foliation maps. Note that the analytic $\Lambda$-Lefschetz
number of the identity map coincides with the $\Lambda$-Euler
characteristic of the foliation introduced by Connes in
\cite{Connes1979,Connes1982}. Connes has also proven the
corresponding Lefschetz trace formula, which is just a measurable
index theorem for the leafwise de Rham complex on a compact foliated
manifold. The topological $\Lambda$-Euler characteristic was studied
recently by Berm\'udez \cite{Bermudez} in a more general context of
measurable foliations.

In \cite{AlvKordy:Betti}, the authors showed that non-vanishing of
the $\Lambda$-Euler characteristic of a transitive codimension one
foliation implies that the reduced leafwise cohomology of this
foliation is infinite-dimensional. This result was extended to the
case of Lie foliations in \cite{AlvKordy} and to the case of
Riemannian foliations in \cite{Mumken}.

Finally, let us mention that Benameur in \cite{Benameur} has proved
a $K$-theoretical version of the Lefschetz formula for leaf
preserving isometries, extending Connes and Skandalis
$K$-theoretical index theorem for foliations \cite{ConnesSkandalis}.

In the rest of the paper, we will only consider the topological
$\Lambda$-Lefschetz number and the topological
$\Lambda$-coincidence. So the word ``topological'' and the subindex
``top'' will be removed from this terminology and notation.

\section{Preliminaries}\label{s:preliminaries}

Recall the following terminology and notation about foliations (see {\em e.g.\/} \cite{CandelConlon}). Let $M$ be a manifold of dimension $n=p+q$. A foliation $\FF$ of {\em dimension\/} $p$ and {\em codimension\/} $q$ on $M$ can be defined as a maximal atlas $\{U_i,\theta_i\}$ of $M$ such that $\theta_i(U_i)=T_i\times B_i$ for some open subset $T_i\subset\R^q$ and some convex open subset $B_i\subset\R^p$, and the corresponding changes of coordinates are of the form
\begin{equation}\label{e:changes of coordinates}
\theta_j\circ\theta_i^{-1}(x,y)=(h_{i,j}(x),g_{i,j}(x,y))\;.
\end{equation}
These charts are called {\em foliation charts\/} of $\FF$. The pair $(M,\FF)$ is called a {\em
foliated manifold\/}. Usually, it is assumed that each $B_i$ is the
standard ball of $\R^p$, but, with the above freedom to choose them,
it will be easier to give foliation charts. Each $U_i$ is called a
{\em distinguished open set\/}. The composite of $\theta_i$ with the
factor projection $T_i\times B_i\to T_i$ is called the {\em local
projection\/} of $\theta_i$ or $U_i$, and denoted by $p_i:U_i\to
T_i$. The fibers of each $p_i$ are called {\em plaques\/}. Each
$T_i$ can be identified to the quotient space of plaques of $U_i$,
which is called the {\em local quotient\/} of $U_i$. The plaques
form a base of a topology, called the {\em leaf topology\/}. The
connected components of the leaf topology are $p$-dimensional
immersed submanifolds of $M$ called {\em leaves\/}. The leaves
determine $\FF$. Let $M/\FF$ denote the quotient space of leaves. A
subset of $M$ is called {\em saturated\/} if it is a union of
leaves.

For each $y\in B_i$, the $q$-dimensional regular submanifold $\theta_i^{-1}(T_i\times\{y\})\subset M$ is called a {\em local transversal\/} of $\FF$. A {\em transversal\/} of $\FF$ is any regular submanifold of $M$ which has an open covering by local transversals. A transversal is called {\em complete\/} if it meets all leaves.

The {\em restriction\/} of $\FF$ to any open subset $U\subset M$ is the foliation $\FF|_U$ whose foliation charts are the foliation charts of $\FF$ defined on subsets of $U$. The leaves of $\FF|_U$ are the connected components of the intersections of the leaves of $\FF$ with $U$.

 A collection of foliation charts whose domains cover $M$ is called a {\em foliation atlas\/}, and determines $\FF$. A foliation atlas $\{U_i,\theta_i\}$ is called {\em regular\/} if $U_i\cup U_j$ is contained in some distinguished open set whenever $U_i$ meets $U_j$. There always exists a locally finite regular atlas.

 For $r\in\N\cup\{\infty\}$, a {\em $C^r$ structure\/} on $\FF$ is a maximal foliation atlas whose changes of coordinates are $C^r$ maps. When $\FF$ is endowed with a $C^r$ structure, then $\FF$ is called a {\em $C^r$ foliation\/}, or it is said that $\FF$ is of {\em class $C^r$\/}.

 If $\FF$ is of class $C^r$ with $r\ge1$, then the vectors tangent to the leaves form a $C^r$ vector subbundle $T\FF\subset TM$, whose fiber at each $x\in M$ is denoted by $T_x\FF$. Then $\nu=TM/T\FF$ is called the {\em normal bundle\/} of $\FF$.

Let $N$ be another $C^r$ manifold, with $r\ge1$. A $C^r$ map $\phi:N\to M$ is said to be {\em transverse\/} to $\FF$ if it is transverse to all leaves of $\FF$; {\em i.e.\/}, when
$$
\phi_*(T_xN)+T_{\phi(x)}\FF=T_{\phi(x)}M
$$
for all $x\in N$. In this case, the inverse image by $\phi$ of the leaves of $\FF$ are the leaves of a foliation $\phi^*\FF$ on $N$, called the {\em pull-back\/} of $\FF$ by $\phi$. If $N$ is a regular $C^r$ submanifold of $M$ and $\phi$ is the inclusion map, then $\phi^*\FF$ is called the {\em restriction\/} of $\FF$ to $N$, and denoted by $\FF|_N$. Observe that the $C^r$ transversals are the $q$-dimensional regular submanifolds transverse to $\FF$.

Let $\{U_i,\theta_i\}$ be any regular foliation atlas of $\FF$ with $\theta_i(U_i)=T_i\times B_i$. Then the maps $h_{i,j}$ of~\eqref{e:changes of coordinates} are local homeomorphisms of $T=\bigsqcup_iT_i$, and thus generate a pseudogroup $\HH$. The pseudogroups generated by different foliation atlases are {\em equivalent\/} in a sense introduced in \cite{Haefliger1980}. The equivalence class of these pseudogroups is called the {\em holonomy pseudogroup\/} of $\FF$. If $\{U_i\}$ is locally finite, then $T$ can be identified to a transversal of $\FF$, and the maps $h_{i,j}$ can be considered as ``slidings'' of local transversals along the leaves. For any transversal of $\FF$, the ``slidings'' of its small enough open sets along the leaves, called {\em holonomy transformations\/}, also generate a representative of the holonomy pseudogroup.

Any $\HH$-invariant measure $\Lambda$ on $T$ is called an {\em
transverse invariant measure\/} of $\FF$. This definition is independent of the
representative $\HH$ of the holonomy pseudogroup because invariant
measures can be pushed forward by pseudogroup equivalences. Thus
$\Lambda$ can be considered as a measure on transversals invariant
by holonomy transformations. There is a maximal saturated open
subset $U\subset M$ such that $\Lambda$ vanishes on transversals
contained in $U$; then $M\setminus U$ is called the {\em support\/}
of $\Lambda$.

Suppose that $\FF$ is $C^1$. Then the class of the Lebesgue measure is well defined on $C^1$ transversals. Thus it makes sense to assume that $\Lambda$ is absolutely continuous with respect to the Lebesgue measure on $C^1$ transversals. This means that $\Lambda$ can be considered as a continuous density on $C^1$ transversals.

{\em Orientations\/} and {\em transverse orientations\/} of
$\FF$ can be defined as orientations of $T\FF$ and $\nu$,
respectively; thus they can be represented by non-vanishing
continuous sections, $\chi$ of $\bigwedge^pT\FF^*$ and $\omega$ of
$\bigwedge^q\nu^*$. Since there is a canonical injection $\nu^*\to
TM^*$, $\omega$ can be considered as a differential $q$-form on $M$.
With the choice of a splitting $TM\cong\nu\oplus T\FF$, we can also
consider $\chi$ as a differential $p$-form on $M$. Moreover
$\omega\wedge\chi$ is independent of the splitting and does not
vanish; thus it defines an induced orientation of $M$. Of course,
orientations and transverse orientations could also be defined
without using differentiability; in particular, a transverse
orientation can be given as an $\HH$-invariant orientation on $T$.

A continuous differential form $\alpha$ on $M$ is called {\em horizontal\/} if $i_X\alpha=0$ for all vector field $X$ tangent to the leaves; {\em i.e.\/}, it can be considered as a continuous section of $\bigwedge\nu^*$. If $\alpha$ is horizontal and invariant by the flow of any $C^1$ vector field $X$ tangent to the leaves, then it is said to be {\em basic\/}. With respect to local foliation coordinates, the coefficients of the expression of a continuous basic form are constant on the plaques. On the other hand, if $\alpha$ is a $C^1$ horizontal $q$-form, then the local expressions of $d\alpha$ involve only leafwise derivatives. Therefore, if $\alpha$ is a continuous basic $q$-form, then there is a sequence of $C^1$ horizontal $q$-forms $\alpha_i$ converging to $\alpha$ such that $d\alpha_i$ converges to zero, where the compact-open topology is considered. Since the basic forms on each $U_i$ are the pull-back of differential forms on $T_i$ by the corresponding local projection, it follows that basic forms can be also considered as $\HH$-invariant differential forms on $T$.

Suppose that $\FF$ is transversely oriented, and thus $T$ has a
corresponding $\HH$-invariant orientation. This orientation can be
used to identified $\HH$-invariant continuous densities on $T$ with
$\HH$-invariant continuous differential forms of degree $q$, which in turn can be considered as continuous basic forms of degree $q$. Therefore any transverse invariant
measure $\Lambda$, absolutely
continuous with respect to the Lebesgue measure, can be identified
to a continuous basic $q$-form $\alpha$. It turns out that, for any
$C^1$ transversal $\Sigma$ and any compactly supported continuous
function $f:\Sigma\to\R$, we have
$$
\int_\Sigma f\,\Lambda=\int_\Sigma f\,\alpha\;,
$$
where the restriction of $\alpha$ to $\Sigma$ is also denoted by $\alpha$. With respect to the opposite orientation on $T$, $\Lambda$ is identified to $-\alpha$.

\section{Spaces of foliation maps}\label{sec:foliation maps}

In this section, we adapt to
foliations the usual strong and weak topologies on the sets of maps between
manifolds, and we shall show some elementary properties of these spaces.

For $r\in\N\cup\{\infty\}$, let $\FF$ and $\FF'$ be $C^r$ foliations on $C^r$ manifolds $M$ and
$M'$ respectively. A {\em foliation map\/} $\phi:(M,\FF)\to(M',\FF')$ (or simply $\phi:\FF\to\FF'$) is a
map $M\to M'$ that maps leaves of $\FF$ to leaves of $\FF'$. If $r\ge1$, for any $C^r$ foliation map $\phi:\FF\to\FF'$, its tangent map $\phi_*:TM\to TM'$ restricts to a $C^r$ homomorphism $\phi_*:T\FF\to T\FF'$.

Let $C^r(\FF,\FF')$ denote the set of $C^r$ foliation maps
$\FF\to\FF'$. The notation $C(\FF,\FF')$ will be also used in the case $r=0$. Assume first that $r$ is
finite, and fix the following data:
\begin{itemize}

\item Any $\phi\in C^r(\FF,\FF')$.

\item Any locally finite
collection of foliation charts  of $\FF$, $\Theta=\{U_i,\theta_i\}$ with
$\theta_i(U_i)=T_i\times B$.

\item A collection of foliation charts of $\FF'$ with the same index set,  $\Theta'=\{U'_i,\theta'_i\}$ with
$\theta'_i(U'_i)=T'_i\times B'$. Let
$\pi'_i:U'_i\to T'_i$ denote the local projection defined by each $\theta'_i$.

\item A family of compact subsets of $M$ with the same index set, $\KK=\{K_i\}$, such that
$K_i\subset U_i$ and $\phi(K_i)\subset U'_i$ for all $i$.

\item A family of positive numbers with the same index set, $\EE=\{\epsilon_i\}$.

\end{itemize}
Then let $\NN^r(\phi,\Theta,\Theta',\KK,\EE)$ be the set of foliation maps
$\psi\in C^r(\FF,\FF')$ such that
\begin{gather}
\psi(K_i)\subset U'_i\;,\label{e:K i}\\
\quad\pi'_i\circ\psi(x)=\pi'_i\circ\phi(x)\;,\label{e:pi'i}\\
\left\|D^k\left(\theta_i'\circ\psi\circ\theta_i^{-1}\right)(\theta_i(x))
-D^k\left(\theta_i'\circ\phi\circ\theta_i^{-1}\right)(\theta_i(x))\right\|<\epsilon_i
\label{e:epsilon i}
\end{gather}
for each $i$, any $x\in K_i$ and every $k\in\{0,\dots,r\}$. Observe
that~\eqref{e:pi'i} means that, for each plaque $P$ of $\theta_i$,
the images of $P\cap K$ by $\phi_i$ and $\psi_i$ lie in the same
plaque of $\theta'_i$. For each finite $r$, the family of such sets
$\NN^r(\phi,\Theta,\Theta',\KK,\EE)$ is a base of a topology on
$C^r(\FF,\FF')$, which will be called the {\em strong
plaquewise topology\/}, or simply the {\em sp-topology\/}. When
$\phi$ is fixed, all possible sets
$\NN^r(\phi,\Theta,\Theta',\KK,\EE)$ form a base of neighbourhoods
of $\phi$ in this space.

On the other hand, the sets $\NN^r(\phi,\Theta,\Theta',\KK,\EE)$,
with $i$ running in a finite index set, form a base of a coarser topology on $C^r(\FF,\FF')$, which
will be called the {\em weak plaquewise topology\/}, or simply the {\em wp-topology\/}. For
families $\Theta$, $\Theta'$, $\KK$ and $\EE$ with just one element, say $\theta$, $\theta'$, $K$ and $\epsilon$, the
corresponding sets $\NN^r(\phi,\theta,\theta',K,\epsilon)=\NN^r(\phi,\Theta,\Theta',\KK,\EE)$ form
a subbase of this topology. When $\phi$ is fixed, all possible sets
$\NN^r(\phi,\theta,\theta',K,\epsilon)$ form a subbase of neighbourhoods of $\phi$ in
this space.

In the case $r=0$, consider also the set $\NN(\phi,\Theta,\Theta',\KK)$ of foliation maps $\psi\in C(\FF,\FF')$ satisfying~(1) and~(2). Observe that all of these sets form a base of the sp-topology of $C(\FF,\FF')$. A base of its wp-topology is similarly given with finite families $\Theta$, $\Theta'$ and $\KK$.

Now let us consider the case $r=\infty$. With the above notation, for arbitrary $s\in\N$ and
$\phi$, $\Theta$, $\Theta'$, $\KK$ and $\EE$ as above, all sets
$$
\cinf(\FF,\FF')\cap\NN^s(\phi,\Theta,\Theta',\KK,\EE)
$$
form a base of a topology on $C^\infty(\FF,\FF')$, which will
be called the {\em strong plaquewise topology\/}, or simply the {\em sp-topology\/}. Finally, the
{\em weak plaquewise topology\/}, or simply the {\em wp-topology\/}, on $C^\infty(\FF,\FF')$ can be
similarly defined by considering families with $i$ running in a finite index set.

For $r=0$, a base of the sp-topology of $C(\FF,\FF')$ is given
by the sets $\NN(\phi,\Theta,\Theta',\KK)$ of foliation maps
$\psi:(M,\FF)\to(M',\FF')$ satisfying~\eqref{e:K i}
and~\eqref{e:pi'i}, where $\phi$, $\Theta$, $\Theta'$ and $\KK$ vary
as above. A base of the wp-topology of $C(\FF,\FF')$ can be
given in the same way with $i$ running in a finite index set.

If $M$ is compact, then both of the above
topologies on each $C^r(\FF,\FF')$ are obviously equal. In this
case, it will be called the {\em plaquewise topology\/} on
$C^r(\FF,\FF')$.

The subindex ``SP'' or ``WP'' will
mean that the restriction of the strong or weak plaquewise topology is considered on any set of $C^r$ foliation maps; the subindex
``P'' will be used when the manifolds are compact. Moreover expressions like ``sp-open,''
``sp-continuous,'' etc. mean that these topological properties hold with respect to the sp-topology
on sets of $C^r$ foliation maps, or for the induced topologies on quotients, products,
etc. A similar notation will be used for the wp-topology. The subindex ``S'' or ``W'' will mean that the usual strong or weak topology is considered on any set of $C^r$ maps.

The following extreme examples may help to clarify the above topologies.

\begin{ex}\label{ex:dim cF' = 0}
If the leaves of $\FF'$ are points ($\dim\FF'=0$),
then $C^r_{SP}(\FF,\FF')$ is discrete.
\end{ex}

\begin{ex}\label{ex:dim cF = 0, dim cF' = dim M'}
If the leaves of $\FF$ are points ($\dim\FF=0$) and the leaves of $\FF'$ are open in $M'$
($\dim\FF'=\dim M'$), then
$C^r_{SP}(\FF,\FF')=C^r_S(M,M')$ and $C^r_{WP}(\FF,\FF')=C^r_W(M,M')$.
\end{ex}

Recall that a subset of a topological space is called {\em residual\/} if it contains a
countable intersection of dense open subsets; the complement of each residual set is
called a {\em meager\/} set. Also, recall that a topological space is called a {\em
Baire space\/} when every residual subset is dense.

\begin{thm}\label{t:Baire}
$C^r_{WP}(\FF,\FF')$ and $C^r_{SP}(\FF,\FF')$ are Baire spaces for $0\le r\le\infty$.
\end{thm}

\begin{proof}
This result could be shown by adapting the proof of the corresponding result for the usual weak
and strong topologies \cite[Chapter~2, Theorem~4.4]{Hirsch}.
But a shorter proof is given by relating the weak and strong plaquewise
topologies with the usual weak and strong topologies.

For $\phi,\Theta,\Theta',\KK$ as above, let $\NN^r(\Theta',\KK)$ denote the set of
foliation maps $\psi\in C^r(\FF,\FF')$ satisfying~\eqref{e:K i} for all $i$,
and let $\NN^r(\phi,\Theta,\Theta',\KK)$ denote the set of
foliation maps $\psi\in C^r(\FF,\FF')$ satisfying~\eqref{e:K i} and~\eqref{e:pi'i} for
all $i$ and $x\in K_i$.
The set $\NN^r(\phi,\Theta,\Theta',\KK)$ is easily seen to be closed in
$\NN^r_W(\Theta',\KK)$. Hence there is a closed subspace $\AAA$ of $C^r_W(\FF,\FF')$
such that
$$
\NN^r(\phi,\Theta,\Theta',\KK)=\AAA\cap\NN^r(\Theta',\KK)\;.
$$
It is easy to check that $C^r(\FF,\FF')$ is closed
in $C^r_W(M,M')$. Then $\AAA$ is closed in $C^r_W(M,M')$ as well, and thus
$\AAA_S$ is a Baire space \cite[Chapter~2, Theorem~4.4]{Hirsch}. On the other hand, since
$C^r_W(M,M')$ is completely metrizable \cite[Chapter~2, Theorem~4.4]{Hirsch}, its closed
subspace $\AAA_W$ is completely metrizable too, and thus $\AAA_W$ is a Baire space by the Baire
Category Theorem.

Since $\NN^r(\Theta',\KK)$ is open in $C^r_S(\FF,\FF')$, it follows that
$\NN^r(\phi,\Theta,\Theta',\KK)$ is open in the Baire space $\AAA_S$. Therefore
$\NN^r_S(\phi,\Theta,\Theta',\KK)$ is a Baire space too \cite[Proposition~8.3]{Kechris}.
Then it easily follows that $C^r_{SP}(\FF,\FF')$ is a Baire space because all
possible sets $\NN^r(\phi,\Theta,\Theta',\KK)$ form an open covering of
$C^r_{SP}(\FF,\FF')$.

Now assume that $i$ runs in a finite index set. Then $\NN^r(\Theta',\KK)$ is also open in
$C^r_W(\FF,\FF')$, and thus $\NN^r(\phi,\Theta,\Theta',\KK)$ is open in the Baire
space $\AAA_W$. Again, it follows that $\NN^r_W(\phi,\Theta,\Theta',\KK)$ is a Baire space,
and thus so is $C^r_{WP}(\FF,\FF')$ because all possible sets
$\NN^r(\phi,\Theta,\Theta',\KK)$, with $i$ running in a finite index set, form an open covering of
$C^r_{WP}(\FF,\FF')$.
\end{proof}

\begin{lem}\label{l:countable}
For $0\le r\le\infty$, any open subset of $C^r_{SP}(\FF,\FF')$ is the
intersection of a countable family of open subsets of
$C^r_{WP}(\FF,\FF')$.
\end{lem}

\begin{proof}
The statement is obvious for the above basic open sets of
$C^r_{SP}(\FF,\FF')$.

An arbitrary open set $\MM$ of
$C^r_{SP}(\FF,\FF')$ is a union of basic open sets $\NN_\alpha$ as
above with
$\alpha$ running in an arbitrary index set. In turn, each $\NN_\alpha$
is the intersection of open subsets $\NN_{\alpha,i}$ of
$C^r_{WP}(\FF,\FF')$ with $i$ running in $\N$. Then
$\MM=\bigcap_{i\in\N}\AAA_i$, where
$$
\AAA_i=\bigcup_\alpha(\NN_{\alpha,0}\cap\dots\cap\NN_{\alpha,i})
$$
is open in $C^r_{WP}(\FF,\FF')$.
\end{proof}

\begin{cor}\label{c:residual}
For $0\le r\le\infty$, any residual subset of $C^r_{SP}(\FF,\FF')$ is also
residual in $C^r_{WP}(\FF,\FF')$.
\end{cor}

\begin{proof}
Let $\AAA$ be a residual subset of $C^r_{SP}(\FF,\FF')$. Thus
$\AAA$ contains a countable intersection of open dense subsets
$\AAA_k$ of $C^r_{SP}(\FF,\FF')$, $k\in\N$. In turn, by
Lemma~\ref{l:countable}, each $\AAA_k$ is a countable intersection of
open subsets $\AAA_{k,\ell}$ of $C^r_{WP}(\FF,\FF')$,
$\ell\in\N$. Moreover each $\AAA_{k,\ell}$ is dense in
$C^r_{WP}(\FF,\FF')$ because it contains $\AAA_k$, which is
dense in $C^r_{WP}(\FF,\FF')$ since the sp-topology is finer
than the wp-topology. Therefore $\AAA$ is residual in
$C^r_{WP}(\FF,\FF')$ because it contains the intersection of the
sets $\AAA_{k,\ell}$.
\end{proof}

The proofs of the following two lemmas are easy exercises, similar
to the proofs of the corresponding results for the usual weak and
strong topologies \cite[page~41]{Hirsch}.

\begin{lem}\label{l:iota*}
Let $U$ be an open subset of $M$, $\FF_U=\FF|_U$, and $\iota:U\to M$
the inclusion map. Then the restriction map
$$
\iota^*:C^r(\FF,\FF')\to C^r(\FF_U,\FF')\;,\quad
\phi\mapsto\phi|_U=\phi\circ\iota\;,
$$
is sp-open for $0\le r\le\infty$.
\end{lem}

\begin{lem}\label{l:cF i}
Let $\{V_i\}$ be a locally finite family of open
subsets of $M$. For each $i$, let $\FF_i=\FF|_{V_i}$, let
$\iota_i:V_i\to M$ denote the inclusion map, and let $\AAA_i$ be a
wp-open subset of $C^r(\FF_i,\FF')$ for $0\le r\le\infty$.
Then $\bigcap_i(\iota_i^*)^{-1}(\AAA_i)$ is sp-open in
$C^r(\FF,\FF')$.
\end{lem}

Let $\FF''$ be another $C^r$ foliation on a $C^r$ manifold $M''$. The proof of the following lemma
is an easy exercise, similar to the proof of the corresponding result for the usual weak and strong
topologies
\cite[page~64--65]{Hirsch}.

\begin{lem}\label{l:composition}
For $0\le r\le\infty$, the following properties hold:
\begin{enumerate}

\item The composition defines a wp-continuous map
$$
C^r(\FF',\FF'')\times C^r(\FF,\FF')\to C^r(\FF;\FF'')\;.
$$

\item If $\sigma:\FF'\to\FF''$ is $C^r$, then the map
$$
\sigma_*:C^r(\FF,\FF')\to C^r(\FF,\FF'')\;,\quad\psi\mapsto\sigma\circ\psi\;,
$$
is sp-continuous.

\item If $\tau:\FF\to\FF'$ is proper and $C^r$, then the map
$$
\tau^*:C^r(\FF',\FF'')\to C^r(\FF,\FF'')\;,\quad
\phi\mapsto\phi\circ\tau\;,
$$
is sp-continuous.

\end{enumerate}
\end{lem}

\begin{lem}\label{l:iota'*}
Let $U'$ be an open subset of $M'$, $\FF'_{U'}=\FF'|_{U'}$, and $\iota':U'\to M'$ the
inclusion map. Then the map
$$
\iota'_*:C^r(\FF,\FF'_{U'})\to C^r(\FF,\FF')\;,\quad
\phi\mapsto\iota'\circ\phi\;,
$$
is an sp-open embedding for $0\le r\le\infty$.
\end{lem}

\begin{proof}
The map $\iota'_*$ is clearly injective, and moreover it is
sp-continuous by Lemma~\ref{l:composition}-(2).

To show that $\iota'_*$ is sp-open, assume first that $r<\infty$.
Then let $\NN$ denote the basic open subset of
$C^r_{SP}(\FF,\FF'_{U'})$ defined with some data
$\phi$, $\Theta$, $\Theta'$, $\KK$ and $\EE$. Since the elements of $\Theta'$ are
also foliation charts of $\FF'$, we can also consider the basic open
subset $\MM$ of $C^r_{SP}(\FF,\FF')$ defined with the data
$\iota'\circ\phi$, $\Theta$, $\Theta'$, $\KK$ and $\EE$. Obviously, $\iota'_*(\NN)$
is the set of maps in $\MM$ whose image is contained in $U'$. But
the image of any map of $\MM$ is contained in $U'$ when $\KK$ covers
$M$, and thus $\iota'_*(\NN)=\MM$ in this case. So $\iota'_*$ is
open as desired.

When $r=\infty$, the result follows similarly by adding an arbitrary $s\in\N$ to the data needed
to define the above basic open sets.
\end{proof}

Let $\FF'_1$ and $\FF'_2$ be $C^r$ foliations on manifolds $M'_1$ and $M'_2$,
respectively, and let
$\FF'_{\times}$ be the {\em product foliation\/} $\FF'_1\times\FF'_2$ on $M'_{\times}=M'_1\times
M'_2$, whose leaves are the products of leaves of
$\FF'_1$ with leaves of $\FF'_2$.

\begin{lem}\label{l:product}
For $0\le r\le\infty$, the canonical map
$$
C^r(\FF,\FF'_1)\times C^r(\FF,\FF'_2)\to
C^r(\FF,\FF'_{\times})\;,
$$
which assigns the mapping $x\mapsto(\phi_1(x),\phi_2(x))$ to each pair
$(\phi_1,\phi_2)$, is a homeomorphism with the weak and strong plaquewise
topologies.
\end{lem}

\begin{proof}
The canonical map of the statement is obviously bijective. Moreover it is easily checked
to be continuous with the weak and strong plaquewise topologies by taking products of foliation
charts of $\FF'_1$ and $\FF'_2$ as foliation charts of $\FF'_{\times}$ to
define basic open sets as above.

The factor projections of the product $M'_1\times M'_2$ are foliation maps
$\pr'_1:\FF'_{\times}\to\FF'_1$ and
$\pr'_2:\FF'_{\times}\to\FF'_2$.
Then the inverse of the canonical map of the statement is given by
$$
\sigma\mapsto(\pr'_1\circ\sigma,\pr'_2\circ\sigma)\;,
$$
which is continuous with the weak and strong plaquewise topologies by
Lemma~\ref{l:composition}-(1)--(2).
\end{proof}

The canonical map of Lemma~\ref{l:product} will be considered as an
identity.

Let $M'_{\times}$ and $\FF'_{\times}$ be as before. Moreover let $\FF_1$ and $\FF_2$ be other $C^r$
foliations on manifolds
$M_1$ and $M_2$, respectively, and let $M_{\times}=M_1\times M_2$ and
$\FF_{\times}=\FF_1\times\FF_2$.

\begin{lem}\label{l:product to product}
For $0\le r\le\infty$, the product map
$$
C^r(\FF_1,\FF'_1)\times C^r(\FF_2,\FF'_2)\to
C^r(\FF_{\times},\FF'_{\times})\;,\qquad
(\phi_1,\phi_2)\mapsto\phi_1\times\phi_2\;,
$$
is a wp-embedding.
\end{lem}

\begin{proof}
This map is obviously injective. Moreover it is wp-continuous by
Lemma~\ref{l:composition}-(1) because, according to Lemma~\ref{l:product},
it can be identified with the map $\pr_1^*\times\pr_2^*$, where
$$
\pr_i^*:C^r(\FF_i,\FF'_i)\to C^r(\FF_{\times},\FF'_i)
$$
is induced by the factor
projection $\pr_i:\FF_{\times}\to\FF_i$ for $i=1,2$.

To finish the proof, assume first that $r$ is finite. For each $i=1,2$, let $\NN_i$ be a
subbasic neighbourhood of some $\phi_i$ in $C^r_{WP}(\FF_i,\FF'_i)$ determined by some data
$\theta_i$, $\theta'_i$, $K_i$ and $\epsilon_i$, as explained in the definition of the
wp-topology. Let $\MM$ be the subbasic neighbourhood of $\phi_1\times\phi_2$ in
$C^r_{WP}(M_{\times},\FF_{\times};M'_{\times},\FF'_{\times})$ defined by the data
$$
\theta_1\times\theta_2\;,\quad\theta'_1\times\theta'_2\;,\quad
K_1\times K_2\;,\quad\min\{\epsilon_1,\epsilon_2\}\;.
$$
It is easy to check that
$$
\psi_1\times\psi_2\in\MM\;\Longrightarrow\;(\psi_1,\psi_2)\in\NN_1\times\NN_2
$$
for any $\psi_i\in C^r(\FF_i,\FF'_i)$, $i=1,2$. Hence the image of
$\NN_1\times\NN_2$ by the product map
is a wp-neighbourhood of $\phi_1\times\phi_2$ in the image of the product map, which finishes the
proof in this case.

If $r=\infty$, the result follows in a similar way by adding an arbitrary $s\in\N$ to the data
needed to define the above subbasic neighbourhoods.
\end{proof}

\begin{lem}\label{l:H(psi)}
For $1\le r<\infty$, let $\FF$ and $\FF'$ be a $C^r$ foliations on manifolds $M$ and $M'$, respectively.
Let $U$ be an open subset of $M$. For each
$C^r$ foliation map
$\phi:\FF\to\FF'$, there is an sp-neighborhood $\MM$ in $C^r(\FF|_U,\FF')$ of the restriction
$\phi|_U$ such that, for each $\psi\in\MM$, the combination $H(\psi)$ of $\psi$ and $\phi|_{M\setminus U}$ is $C^r$ foliation map $\FF\to\FF'$, and moreover the assignment $\psi\mapsto H(\psi)$
defines a continuous map $H:\MM_{SP}\to C^r_{SP}(\FF,\FF')$.
\end{lem}

\begin{proof}
This is analogous to the proof of Lemma~2.8 of Chapter~2 in~\cite{Hirsch}.
Let $\{V_k\}_{k\in\N}$ be an open covering of $U$ such that $\overline{V_k}\subset V_{k+1}$. Consider some data $\Theta$, $\Theta'$ and $\KK$, like in the definition of the sp-topology, such that $\{U_i\}$ is an open covering of $U$. For each $i$, let
$$
\epsilon_i=\frac{1}{\min\{k\in\N\ |\ K_i\subset V_k\}}\;,
$$
and set $\EE=\{\epsilon_i\}$. Then the result follows with $\MM=\NN^r(\phi|_U,\Theta,\Theta',\KK,\EE)$.
\end{proof}

When a foliation map $\phi:\FF\to\FF$ preserves each leaf, then it will
be called a {\em leaf preserving map\/}. For $0\le r\le\infty$, the family of $C^r$ leaf preserving maps $\FF\to\FF$ will be
denoted by $C^r_{\text{\rm leaf}}(\FF,\FF)$. The proof of the following
result is elementary.

\begin{prop}\label{p:C r leaff(cF,cF) is open}
The set $C^r_{\text{\rm leaf}}(\FF;\FF)$ is open in
$C^r_{SP}(\FF,\FF)$.
\end{prop}

\begin{cor}\label{c:C r leaff(cF,cF) is Baire}
$C^r_{\text{\rm leaf},SP}(\FF;\FF)$ is a Baire space.
\end{cor}

\begin{proof}
By \cite[Proposition~8.3]{Kechris}, this is a direct consequence of
Theorem~\ref{t:Baire} and Proposition~\ref{p:C r leaff(cF,cF) is open}.
\end{proof}

Let $\Diff^r(M)$ be the set of $C^r$
diffeomorphisms $M\to M$; here, a $C^0$ diffeomorphism is just a homeomorphism.
Then let
\begin{align*}
\Diff^r(M,\FF)&=\Diff^r(M)\cap C^r(\FF,\FF)\;,\\
\Diff^r(\FF)&=\Diff^r(M)\cap C^r_{\text{\rm leaf}}(\FF;\FF)\;.
\end{align*}
With the operation of composition, $\Diff^r(M,\FF)$ is a group, and
$\Diff^r(\FF)$ is a normal subgroup of $\Diff^r(M,\FF)$. Assuming
$r=\infty$, this notation is compatible with the usual notation
$\mathfrak{X}(M,\FF)$ for the Lie algebra of infinitesimal
transformations of $(M,\FF)$, and
$\mathfrak{X}(\FF)\subset\mathfrak{X}(M,\FF)$ for the normal Lie
subalgebra of vector fields tangent to the leaves. The flow of any
vector field in $\mathfrak{X}(M,\FF)$ (respectively,
$\mathfrak{X}(\FF)$) is a uniparametric subgroup of
$\Diff^\infty(M,\FF)$ (respectively, $\Diff^\infty(\FF)$).

\begin{cor}\label{c:Diff is open}
For $1\le r\le\infty$, the sets $\Diff^r(M,\FF)$ and
$\Diff^r(\FF)$ are open in
$C^r_{SP}(\FF,\FF)$.
\end{cor}

\begin{proof}
This is a consequence of Proposition~\ref{p:C r leaff(cF,cF) is open}
since $\Diff^r(M)$ is open in $C^r_S(M,M)$
\cite[page~38]{Hirsch}.
\end{proof}

\begin{cor}\label{c:Diff is Baire}
$\Diff^r_{SP}(M,\FF)$ and $\Diff^r_{SP}(\FF)$ are Baire spaces for $1\le r\le\infty$.
\end{cor}

\begin{proof}
This follows from \cite[Proposition~8.3]{Kechris}, Theorem~\ref{t:Baire}
and Corollary~\ref{c:Diff is open}.
\end{proof}

The proof of the following lemma is an easy exercise, similar to the proof of the
corresponding result for the usual weak and strong topologies
\cite[page~64]{Hirsch}.

\begin{lem}\label{l:Diff is a topological group}
For $0\le r\le\infty$, $\Diff^r(M,\FF)$ and $\Diff^r(\FF)$ are topological groups with the operation of
composition and the weak and strong plaquewise topologies.
\end{lem}

\section{Leafwise approximations}\label{s:approximations}

The sp-topology is well
adapted to improve leafwise differentiability by approximation of foliation maps. We need
certain result of this kind.

For $r,s\in\N\cup\{\infty\}$ with $s<r$, let $\FF$ and $\FF'$ be $C^r$
foliations on manifolds
$M$ and $M'$, respectively. A foliation map $\phi:\FF\to\FF'$
will be called of {\em class $C^{r,s}$\/} if it is of class $C^s$
and moreover, for any pair of domains $U$ and $U'$ of foliation charts of
$\FF$ and $\FF'$ with $\phi(U)\subset U'$, the composite of $\phi:U\to U'$
with the local projection $U'\to T'$ is $C^r$, where $T'$ is the
local quotient of $U'$. The set of $C^{r,s}$ foliation maps $\FF\to\FF'$ will be
denoted by $C^{r,s}(\FF,\FF')$; the notation
$C^{r,s}_{SP}(\FF,\FF')$ means that this set is endowed with
the restriction of the sp-topology of $C^s(\FF,\FF')$.

\begin{ex}\label{ex:local}
Let $T$, $T'$, $L$ and $L'$ be open subsets of Euclidean spaces. Assume that
$L$ and $L'$ are connected, and let $\FF$ and $\FF'$ be the \cinf\ foliations on
$M=T\times L$  and $M'=T'\times L'$ with leaves $\{x\}\times L$ and
$\{x'\}\times L'$ for $x\in T$ and $x'\in T'$, respectively. For any
$u\in C^s(\FF,\FF')$, write $u(x,y)=(\bar u(x),\tilde u(x,y))$
for all $(x,y)\in N$. Then the foliation map $u$ is of class
$C^{r,s}$ if and only if the map $\bar u:T\to T'$ is of class $C^r$.
\end{ex}

\begin{ex}
This example shows that $C_{\text{\rm leaf}}^s(\FF,\FF)\not\subset C^{r,s}(\FF,\FF)$ in general. Let $M$ be the M\"obius band without boundary, considered as the quotient of $\R^2$ by the group of transformations generated by the diffeomorphism $\sigma$ defined by $\sigma(x,y)=(-x,y+1)$. The $C^\infty$ foliation $\widetilde{\FF}$ of $\R^2$ with leaves $\{x\}\times\R$, $x\in\R$, is $\sigma$-invariant. Thus $\widetilde{\FF}$ induces a $C^\infty$ foliation $\FF$ on $M$. The foliation map $\tilde\phi:\widetilde{\FF}\to\widetilde{\FF}$, defined by $\tilde\phi(x,y)=(|x|,0)$, satisfies $\tilde\phi\circ\sigma=\tilde\phi$. So $\tilde\phi$ induces a foliation map $\phi:\FF\to\FF$, which is $C^{0,\infty}$ and leaf preserving, but it is not $C^1$.
\end{ex}

\begin{thm}[Leafwise approximation theorem]\label{t:approximation}
With the above notation, suppose that some $\phi\in
C^{r,s}(\FF,\FF')$ is $C^r$ on some neighbourhood of a closed
subset $E\subset M$. Then any neighbourhood of $\phi$ in
$C^{r,s}_{SP}(\FF,\FF')$ contains some $\psi\in
C^r(\FF,\FF')$ which equals $\phi$ on some neighbourhood of
$E$. In particular, the set $C^r(\FF,\FF')$ is dense in
$C^{r,s}_{SP}(\FF,\FF')$.
\end{thm}

The first step to prove Theorem~\ref{t:approximation} is the following local result.

\begin{lem}\label{l:approximation}
With the notation of Example~\ref{ex:local}, let
$Q\subset M$ be a closed subset and $W\subset M$ an open subset. If
some $u\in C^{r,s}(\FF,\FF')$ is of class $C^r$ on some
neighbourhood of $Q\setminus W$, then any neighbourhood of $u$ in
$C^{r,s}_{SP}(\FF,\FF')$ contains some map $v$ which is $C^r$
on some neighbourhood of $Q$, and which equals $u$ on $M\setminus
W$.
\end{lem}

\begin{proof}
Fix any neighbourhood $\NN$ of $u$ in $C^{r,s}_{SP}(\FF,\FF')$.
Let $\MM^s(u)$ be the set of $C^s$ foliation maps
$v:\FF\to\FF'$ such that $\bar v=\bar u$.
Note that $\MM^s(u)$ is open in $C^s_{SP}(\FF,\FF')$. So $\MM^s(u)$ is open in
$C^{r,s}_{SP}(\FF,\FF')$ too, and thus we can assume that $\NN\subset\MM^s(u)$.
Moreover the spaces $C^s_S(M,M')$ and
$C^s_{SP}(\FF,\FF')$ induce the same topology on $\MM^s(u)$, and the corresponding
space will be denoted by
$\MM^s_S(u)$. Observe also that the second factor projection, $\pr_2:M'\to L'$, induces a
homeomorphism
$$
\pr_{2*}:\MM^s_S(u)\to C^s_S(M,L')\;,\quad v\mapsto\pr_2\circ v=\tilde v\;,
$$
Therefore, by the relative approximation theorem \cite[pp.~48--49]{Hirsch}, there is some
$h\in\pr_{2*}(\NN)$ which is $C^r$ on some neighbourhood of $Q$, and which equals $\tilde u$ on
$M\setminus W$. Then the result follows with the map $v\in\MM^s(u)$ satisfying $\tilde v=h$.
\end{proof}

\begin{proof}[Proof of Theorem~\ref{t:approximation}]
This is similar to the proof of the usual approximation theorem \cite[p.~49]{Hirsch}.
Fix the following:
\begin{itemize}

\item A locally finite open covering $\{U_i\}$ of $M$ by relatively compact
domains of foliation charts of $\FF$.

\item A covering $\{K_i\}$ of $M$ by compact subsets satisfying $K_i\subset U_i$
for all $i$.

\item An open subset $V_i\subset M$ for each $i$ with $K_i\subset V_i$ and
$\overline{V_i}\subset U_i$.

\item For each $i$, a domain $U'_i$ of some foliation chart of
$\FF'$ such that $\phi\left(\overline{U_i}\right)\subset U'_i$. Let
$\pi'_i:U'_i\to T'_i$ denote the local projection of $U'_i$.

\item An open neighbourhood $O$ of $E$ so that $\phi$ is $C^r$ around $\overline{O}$.

\item Any neighbourhood $\MM$ of $\phi$ in $C^{r,s}_{SP}(\FF,\FF')$ satisfying
$\psi\left(\overline{U_i}\right)\subset U'_i$ for all $\psi\in\MM$ and all $i$; this
property holds when $\MM$ is small enough.

\end{itemize}
Since the open covering $\{U_i\}$ of $M$ is locally
finite, $i$ runs in a countable index set; say either in $\{0,\dots,k\}$ for some
$k\in\N$, or in $\N$.

\begin{claim}\label{cl:psi i}
For each $i$, there is some $\psi_i\in\MM$ which is $C^r$ on some
neighbourhood of $K_0\cup\dots\cup K_i$, which equals $\phi$ on some neighbourhood of
$\overline{O}$, and which equals $\psi_{i-1}$ on some neighbourhood of $M\setminus U_i$ if $i>0$.
\end{claim}

To prove this assertion, each $\psi_i$ is defined by induction on $i$. To
simplify the construction of this sequence, we can assume $U_0=K_0=\emptyset$, and take
$\psi_0=\phi$. Now let $i>0$, and assume that $\psi_{i-1}$ is defined. Let
$\FF_i=\FF|_{U_i}$ and
$\FF'_i=\FF'|_{U'_i}$, and let $u_{i-1}$ denote the restriction
$\psi_{i-1}:\FF_i\to\FF'_i$, which is of class $C^{r,s}$. Let
$\YY_i$ denote the subspace of $C^{r,s}_{SP}(\FF_i,\FF'_i)$
whose elements are the foliation maps
$v\in C^{r,s}(\FF_i,\FF'_i)$ satisfying
$$
\pi'_i\circ v=\pi'_i\circ u_{i-1}\;,\quad
v|_{U_i\setminus V_i}=u_{i-1}|_{U_i\setminus V_i}\;.
$$
Let $H_i:\YY_i\to C^{r,s}_{SP}(\FF,\FF')$ be the extension map given by
$$
H_i(v)=
\begin{cases}
v&\text{on $U_i$}\\
\psi_{i-1}&\text{on $M\setminus U_i$}\;.
\end{cases}
$$
It is easy to prove that $H_i$ is continuous, and thus there is some neighbourhood $\NN_i$ of
$u_{i-1}$ in $\YY_i$ such that $H_i(\NN_i)\subset\MM$. Let $\NN'_i$ be a neighbourhood of
$u_{i-1}$ in $C^{r,s}_{SP}(\FF_i,\FF'_i)$ with $\NN_i=\YY_i\cap\NN'_i$.
Then, by Lemma~\ref{l:approximation}, there
is some $v_i\in\NN'_i$ which is
$C^r$ on some neighbourhood of $K_i$, and which equals $u_{i-1}$ on $U_i\setminus W_i$, where
$W_i=V_i\cap\left(U_i\setminus\overline{O}\right)$. It follows that $v_i\in\YY_i$, and
Claim~\ref{cl:psi i} holds with $\psi_i=H_i(v_i)$.

The result follows from Claim~\ref{cl:psi i} because a foliation map $\psi\in\MM$ is well
defined by
$\psi(x)=\psi_{i_x}(x)$ for each $x\in M$, where
$$
i_x=\max\{i\ |\ x\in U_i\}\;.
$$
Such a $\psi$ is $C^r$ because the sets $K_i$ cover $M$, and $\psi$ is obviously equal to
$\phi$ on $O$.
\end{proof}

\section{Transversality for foliation maps}\label{s:transver}

For $1\le r\le\infty$, let $M$ and $M'$ be $C^r$ manifolds, and let $A$ be a $C^r$ embedded
submanifold of $M'$. Recall that a $C^r$ map $\phi:M\to M'$ is {\em transverse\/} to $A$ along a
subset $K\subset M$ when
$$
T_{\phi*(x)}M'=T_{\phi*(x)}A+\phi_*(T_xM)
$$
for all $x\in K$. The standard notation
$\phi\pitchfork_KA$ will be used in this case. If $\phi\pitchfork_MA$, then it is simply said that $\phi$ is {\em transverse\/} to $A$, and the notation $\phi\pitchfork A$ is used. As usual, let
$$
\pitchfork^r\!\!(M,M';A)=
\left\{\phi\in C^r(M,M')\ |\ \phi\pitchfork A\right\}\;.
$$

Now, let $\FF$ and $\FF'$ be $C^r$ foliations on $M$ and $M'$, respectively. In general, with the sp-topology, it may not
be possible to approach any foliation map $\FF\to\FF'$ by foliation
maps transverse to $A$, as can be easily seen by using
Example~\ref{ex:dim cF' = 0}. But the next theorem shows that this type
of approach holds when
$A$ is transverse to $\FF'$; \textit{i.e.},
$$
T_{x'}M'=T_{x'}A+T_{x'}\FF'
$$
for all $x'\in A$. This is equivalent to the existence of a foliation chart $\theta':U'\to T'\times B'$
of $\FF'$ around each point of $A$ so that
$$
\theta'(T'\times B'_0)=A\cap U'
$$
for some $C^r$ embedded submanifold $B'_0\subset B'$.

\begin{thm}[Foliations Transversality Theorem]\label{t:transver}
With the above notation, if $A$ is transverse to $\FF'$, then
$$
\pitchfork^r\!\!(M,M';A)\cap C^r(\FF,\FF')
$$
is residual in $C^r_{SP}(\FF,\FF')$.
\end{thm}

As in the proof of the usual transversality theorem
\cite[pp.~74--77]{Hirsch}, a local version is shown first, and then
a globalization procedure is applied. The notation of
Lemma~\ref{l:approximation} is used to prove this local result.
Thus, again, let $T$, $T'$, $L$ and $L'$ be open subsets of Euclidean spaces,
where $L$ and $L'$ are connected, and let $\FF$ and $\FF'$ be the foliations on
$M=T\times L$  and $M'=T'\times L'$ with leaves $\{x\}\times L$ and
$\{x'\}\times L'$ for $x\in T$ and $x'\in T'$, respectively. Let
$L'_0\subset L'$ be an embedded submanifold, and set $A=T'\times
L'_0\subset M'$.

\begin{lem}\label{l:transver}
The set
$$
\pitchfork^r\!\!(M,M';A)\cap C^r(\FF,\FF')
$$
is residual in $C^r_{SP}(\FF,\FF')$.
\end{lem}

\begin{proof}
For any $u\in C^r(\FF,\FF')$, consider the space $\MM^r_S(u)$ introduced in
the proof of Lemma~\ref{l:approximation}, as well as the homeomorphism
$$
\pr_{2*}:\MM^r_S(u)\to C^r_S(M,L')\;,
$$
where $\pr_2:M'\to L'$ denotes the second factor projection. Observe that
$$
\pr_{2*}^{-1}\left(\pitchfork^r\!\!(M,L';W)\right)=
\text{}\pitchfork^r\!\!(M,M';A)\cap\MM^r(u)
$$
since $A=T'\times L'_0$. So
$$
\pitchfork^r\!\!(M,M';A)\cap\MM^r(u)
$$
is residual in $\MM^r_S(u)$ since $\pitchfork^r\!\!(M,L';L'_0)$ is
residual in $C^r_S(M,L')$ by the transversality theorem. Then the result follows since the
sets $\MM^r_S(u)$ form an open covering of $C^r_{SP}(\FF,\FF')$ when $u$
runs in $C^r(\FF,\FF')$.
\end{proof}

Now, we show a general globalization procedure for foliations. For $C^r$
foliated manifolds
$(M,\FF)$ and $(M',\FF')$, a {\em $C^r$ mapping class\/} on $(\FF,\FF')$ is
an assignment of a subset
$$
\XX(U,U')\subset C^r(\FF|_U,\FF'|_{U'})
$$
to each pair of open subsets, $U\subset M$ and $U'\subset M'$, such that the following {\em
localization axiom\/} is satisfied: Given any $\phi\in
C^r(\FF|_U,\FF'|_{U'})$ for open subsets $U\subset M$ and $U'\subset M'$, we have
$\phi\in\XX(U,U')$ if there exists an open covering $\{U_i\}$ of $U$, and a family $\{U'_i\}$ of open
subsets of $U'$ such that the restriction $\phi:U_i\to U'_i$ is well
defined and belongs to $\XX(U_i,U'_i)$ for every $i$.

A $C^r$ mapping class $\XX$ on $(\FF,\FF')$ is called {\em
rich\/} if there are bases $\UU$ and $\UU'$ of open sets of $M$ and $M'$ such
that $\XX(U,U')$ is sp-residual in $C^r(\FF|_U,\FF'|_{U'})$ for
all $U\in\UU$ and $U'\in\UU'$. This definition does not correspond
with the usual concept of a rich mapping class on manifolds
\cite[page~75]{Hirsch}, but it is appropriate for the type of arguments
that will be used in this paper.

\begin{thm}[Foliations Globalization Theorem]\label{t:global}
If $\XX$ is a rich $C^r$ mapping class on $(\FF,\FF')$, then
$\XX(M,M')$ is residual in $C^r_{SP}(\FF,\FF')$.
\end{thm}

\begin{proof}
Let $\UU$ and $\UU'$ be bases of open sets of $M$ and $M'$ satisfying the richness condition of
$\XX$. Take (indexed) subfamilies
$$
\VV=\{V_i\}\subset\UU\;,\quad
\VV'=\{V'_i\}\subset\UU'
$$
such that $\VV$ is a locally finite covering of $M$, and
$\overline{V_i}$ is compact for each $i$. In particular, the index $i$ runs in a
countable set. Then let
$$
\NN^r(\VV,\VV')=\left\{\phi\in C^r(\FF,\FF')\ |\
\text{$\phi\left(\overline{V_i}\right)\subset V'_i$ for all $i$}\right\}\;.
$$
Observe that $\NN^r(\VV,\VV')$ is an open subset of $C^r_{SP}(\FF,\FF')$, and thus, by
Theorem~\ref{t:Baire} and \cite[Proposition~8.3]{Kechris}, it is a Baire space with the sp-topology.

For each $i$, let
$\FF_i=\FF|_{V_i}$ and $\FF'_i=\FF'|_{V'_i}$, and let $\iota_i$ and $\iota'_i$
denote the inclusion maps of $V_i\hookrightarrow M$ and $V'_i\hookrightarrow M'$, which induce the maps
\begin{gather*}
\iota_i^*:C^r(\FF,\FF')\to C^r(\FF_i,\FF')\;,\quad
\phi\mapsto\phi|_{V_i}=\phi\circ\iota_i\;,\\
\iota'_{i*}: C^r(\FF_i,\FF'_i)\to C^r(\FF_i,\FF')\;,\quad
\phi\mapsto\iota'_i\circ\phi\;.
\end{gather*}

Each set $\XX(V_i,V'_i)$ is residual in
$C^r_{SP}(\FF_i,\FF'_i)$, and thus contains
a countable intersection of open dense
subsets $\AAA_{i,k}$ of $C^r_{SP}(\FF_i,\FF'_i)$, $k\in\N$. Each
$\iota'_{i*}(\AAA_{i,k})$ is open in $C^r_{SP}(\FF_i,\FF')$ by
Lemma~\ref{l:iota'*}. So
$\iota'_{i*}(\AAA_{i,k})$ is equal to a countable intersection of open
subsets $\BB_{i,k,\ell}$ of $C^r_{WP}(\FF_i,\FF')$, $\ell\in\N$, by
Lemma~\ref{l:countable}. Moreover each set
$$
\CC_{k,\ell}=\bigcap_i(\iota_i^*)^{-1}(\BB_{i,k,\ell})\;,
$$
is open in $C^r_{SP}(\FF,\FF')$ by Lemma~\ref{l:cF i}.

Since each $\AAA_{i,k}$ is dense in
$C^r_{SP}(\FF_i,\FF'_i)$, each $\iota'_{i*}(\AAA_{i,k})$ is also
dense in the open subspace $\iota'_{i*}(C^r(\FF_i,\FF'_i))$ of
$C^r_{SP}(\FF_i,\FF')$ by Lemma~\ref{l:iota'*}. Since
$$
\iota_i^*(\NN^r(\VV,\VV'))\subset\iota'_{i*}(C^r(\FF_i,\FF'_i))\;,
$$
it follows that each set
$$
\NN^r(\VV,\VV')\cap(\iota_i^*)^{-1}(\iota'_{i*}(\AAA_{i,k}))
$$
is dense in $\NN^r_{SP}(\VV,\VV')$ by Lemma~\ref{l:iota*}. Therefore
$$
\NN^r(\VV,\VV')\cap(\iota_i^*)^{-1}(\BB_{i,k,\ell})
$$
is a dense open subset of $\NN^r_{SP}(\VV,\VV')$ by
Lemma~\ref{l:composition}, and thus $\NN^r(\VV,\VV')\cap\CC_{k,\ell}$ is residual
in the Baire space $\NN^r_{SP}(\VV,\VV')$ because $i$ runs in a countable index set.
So $\NN^r(\VV,\VV')\cap\CC_{k,\ell}$ is a dense open subset of
$\NN^r_{SP}(\VV,\VV')$.
Moreover
\begin{align*}
\bigcap_{k,\ell=1}^\infty\CC_{k,\ell}
&=\bigcap_i\bigcap_{k,\ell=1}^\infty(\iota_i^*)^{-1}(\BB_{i,k,\ell})\\
&=\bigcap_i\bigcap_{k=1}^\infty(\iota_i^*)^{-1}(\iota'_{i*}(\AAA_{i,k}))\\
&\subset\bigcap_i(\iota_i^*)^{-1}(\iota'_{i*}(\XX(V_i,V'_i)))\\
&\subset\XX(M,M')\;,
\end{align*}
where the last inclusion holds by the localization axiom. Therefore
$$
\XX(M,M')\cap\NN^r(\VV,\VV')
$$
is residual in $\NN^r_{SP}(\VV,\VV')$, yielding that $\XX(M,M')$ is residual in
$C^r_{SP}(\FF,\FF')$ because all possible sets $\NN^r(\VV,\VV')$ form an open covering of
$C^r_{SP}(\FF,\FF')$.
\end{proof}

\begin{proof}[Proof of Theorem~\ref{t:transver}]
Let $\XX$ be defined by
$$
\XX(U,U')=\text{}\pitchfork^r\!\!(U,U';A\cap U')\cap
C^r(\FF|_U,\FF'|_{U'})
$$
for open subsets $U\subset M$ and $U'\subset M'$. Obviously, $\XX$ satisfies the localization axiom, and
thus it is a $C^r$ mapping class on $(\FF,\FF')$. Moreover, by
Lemma~\ref{l:transver}, $\XX$ satisfies the richness condition with the bases of open sets
given by domains of foliation charts of $\FF$ and $\FF'$. Then the result follows by
Theorem~\ref{t:global}.
\end{proof}

\section{Local saturation and clean intersection}\label{s:clean}
For $0\le r\le\infty$, let $\FF$ be a $C^r$ foliation on a manifold
$M$, and let $A$ be a subset of $M$. A {\em local saturation\/} of
$A$ is the saturation of $A$ with respect to the restriction of
$\FF$ to some open neighbourhood of $A$ in $M$. A set is called {\em
locally saturated\/} if it is a local saturation of itself.

Observe that, if $A$ is a closed regular $C^r$
submanifold of $M$ transverse to $\FF$, then a local saturation of
$A$ is an arbitrary open neighborhood $U$ of $A$ such that each leaf
of $\FF|_U$ meets $A$.

The {\em intersection\/} of $A$ and $\FF$, denoted by $A\cap\FF$, is
the partition of $A$ into the intersections of $A$ with the leaves
of $\FF$. If the connected components of the classes of $A\cap\FF$
are leaves of a foliated structure on $A$ (see {\em e.g.\/}
\cite{CandelConlon}), this foliated structure is called the {\em
restriction\/} of $\FF$ to $A$, and denoted by $\FF|_A$. For instance,
this restriction is defined when $A$ is locally saturated.

Assume now that $A$ is a regular $C^r$ submanifold of $M$. The
intersection $A\cap\FF$ is called {\em clean\/} if there exists
a foliation chart $\theta:U\to T\times B$ of $\FF$ around each point
of $A$ such that
$$
\theta(U\cap A)=T_0\times B_0
$$
for some regular $C^r$ submanifolds $T_0\subset T$ and $B_0\subset
B$. In this case, the connected components of the classes of $A\cap
L$ are the leaves of a $C^r$ foliation $\FF|_A$ on $A$, which is
also called the {\em restriction\/} of $\FF$ to $A$. It is clear that, if $A$ is transverse to $\FF$, then $A\cap\FF$ is clean. In this case, $T_0$ coincides with $T$.

\begin{prop}\label{p:clean}
If $A\cap\FF$ is clean, then any
neighbourhood of $A$ contains some local saturation $S$ of $A$ such
that $S$ is a regular $C^r$ submanifold of $M$, $\FF|_S$ is a
$C^r$ foliation whose leaves are open subsets of the leaves of $\FF$, and the inclusion map $T_x\FF\hookrightarrow T_xS$ induces an isomorphism
\begin{equation}\label{e:T xFF/T x(FF| A)}
T_x\FF/T_x(\FF|_A)\cong T_xS/T_xA
\end{equation}
for each $x\in A$.
\end{prop}

\begin{proof}
In any neighborhood of $A$, there is a locally finite family of
foliation charts, $\{U_i,\theta_i\}$, such that each $\theta_i$ satisfies the condition of clean
intersection for $A$ and $\FF$, $A\subset\bigcup_iU_i$, and $U_i\cup U_j$ is contained in the domain
of another foliation chart satisfying the condition of clean
intersection whenever $U_i\cap U_j\neq\emptyset$. Then
the statement holds with $S$ equal to $\FF|_V$-saturation of $A$, where $V=\bigcup_iU_i$.
\end{proof}

Suppose that $A\cap\FF$ is clean. A local saturation is called {\em
good\/} if it satisfies the properties stated in
Proposition~\ref{p:clean}. Let $C^r(\FF,\FF';A)$ denote the set of
maps $\phi\in C^r(\FF,\FF')$ such that there is some good local
saturation $S$ of $A$ so that $\phi^{-1}(S)$ is a neighborhood of
$\phi^{-1}(A)$ in $M$; the notation $C(\FF,\FF';A)$ is also used
when $r=0$. For $0\le s<r$, let
$$
C^{r,s}(\FF,\FF';A)=C^{r,s}(\FF,\FF')\cap C(\FF,\FF';A)\;.
$$

The intersection of two good local saturations of $A$ is a good local saturation of $A$. Moreover, if $S$ is a fixed good local saturation of $A$, the good local saturations of $A$ contained in $S$ form a base of open neighborhoods of $A$ in $S$. Hence, if $S$ is some fixed good local saturation of $A$, then $\phi^{-1}(S)$ is a neighborhood of $\phi^{-1}(A)$ in $M$ for all $\phi\in C^r(\FF,\FF';A)$. This easily yields the following result.

\begin{prop}\label{p:C r(cF,cF';A) is sp-open}
$C^r(\FF,\FF';A)$ is open in $C^r_{SP}(\FF,\FF')$.
\end{prop}

\begin{cor}\label{c:C r SP(cF,cF';A) is Baire}
$C^r_{SP}(\FF,\FF';A)$ is a Baire space.
\end{cor}

\begin{proof}
This is a consequence of \cite[Proposition~8.3]{Kechris}, Theorem~\ref{t:Baire} and Proposition~\ref{p:C r(cF,cF';A) is sp-open}.
\end{proof}

\begin{cor}\label{c:C r(cF,cF';A) is dense in C r,s(cF,cF';A)}
For $0\le s<r$,  suppose that some $\phi\in
C^{r,s}(\FF,\FF';A)$ is $C^r$ on some neighbourhood of a closed
subset $E\subset M$. Then any neighbourhood of $\phi$ in
$C^{r,s}_{SP}(\FF,\FF';A)$ contains some $\psi\in
C^r(\FF,\FF';A)$ which equals $\phi$ on some neighbourhood of
$E$. In particular, the set $C^r(\FF,\FF';A)$ is dense in
$C^{r,s}_{SP}(\FF,\FF';A)$.
\end{cor}

\begin{proof}
This follows from Theorem~\ref{t:approximation} and Proposition~\ref{p:C r(cF,cF';A) is sp-open}.
\end{proof}

\section{Transversality in good local saturations}\label{s:ls-transver}

For $1\le r\le\infty$, let $\FF$ and $\FF'$ be $C^r$ foliations on manifolds $M$ and $M'$, and let $A$ be
a regular $C^r$ submanifold of $M'$ whose intersection with $\FF'$
is clean. A foliation map
$\phi\in C^r(\FF,\FF')$ will be called {\em ls-transverse\/} to $A$ on a subset $K\subset M$,
denoted by
$\phi\pitchfork_{\text{\rm ls},K}A$, when there is some open neighbourhood $U$ of $\phi^{-1}(A)$
and some good local saturation $S$ of $A$ such that
$\phi(U)\subset S$ and the restriction $\phi:U\to S$ is transverse to $A$ on $K$. If $K=M$, then it is said that $\phi$ is {\em ls-transverse\/} to $A$ and the notation $\phi\pitchfork_{\text{\rm ls}}A$ is used. Here, ``ls-transverse'' stands for ``transverse in good local saturations''.

\begin{prop}\label{p:ls-transver}
If $\phi\pitchfork_{\text{\rm ls}}A$, then $\phi^{-1}(A)$ is a regular $C^r$ submanifold of $M$ whose codimension equals the codimension of $A$ in any of its good local saturations.
\end{prop}

\begin{proof}
With the above notation, since $\phi:U\to S$ is transverse to $A$, we get that $\phi^{-1}(A)$ is a regular $C^r$ submanifold of $U$ whose codimension is equal to the codimension of $A$ in $S$ (see {\em e.g.\/} Theorem~3.3 of Chapter~1 in \cite{Hirsch}). Then the result follows because $U$ is open in $M$.
\end{proof}

Let
$$
\pitchfork_{\text{\rm ls}}^r\!\!(\FF,\FF';A)= \{\phi\in
C^r(\FF,\FF')\ |\ \phi\pitchfork_{\text{\rm ls}}A\}\;.
$$
We will need the
following slight sharpening of Theorem~\ref{t:transver}.

\begin{thm}\label{t:ls-transverse}
With the above notation, if $A$ is closed in $M'$, then $\pitchfork_{\text{\rm
ls}}^r\!\!(\FF,\FF';A)$ is residual in $C^r_{SP}(\FF,\FF';A)$.
\end{thm}

\begin{proof}
For any open subset $U\subset M$ and any good local saturation $S$ of $A$, it is easy to check
that the set
$$
\NN^r(U,S)=\{\phi\in C^r(\FF,\FF')\ |\
\phi(\overline{U})\subset S,\ \phi(M\setminus U)\subset M'\setminus A\}
$$
is open in $C^r_{SP}(\FF,\FF';A)$ because $A$ is closed in $M'$. Let
$\iota:U\to M$ and $\iota':S\to M'$ denote the inclusion maps, and let
$\FF_U=\FF|_U$ and $\FF'_S=\FF'|_S$. Let $U'$ be an open subset of $M'$ such that $S$ is the
$\FF'_{U'}$-saturation of $A$, where
$\FF'_{U'}=\FF'|_{U'}$.
Observe that $C^r(\FF_U,\FF'_S)$ is sp-open in
$C^r(\FF_U,\FF'_{U'})$, and thus the map
$$
\iota'_*:C^r(\FF_U,\FF'_S)\to C^r(\FF_U,\FF')
$$
is an sp-open embedding by Lemma~\ref{l:iota'*}.

The set
$$
\XX(U,S;A)=\pitchfork^r\!\!(U,S;A)\cap C^r(\FF_U,\FF'_S)
$$
is sp-residual in $C^r(\FF_U,\FF'_S)$ by Theorem~\ref{t:transver}. Then
$\XX(U,S;A)$ contains a countable intersection of open dense subsets $\AAA_k$ of
$C^r_{SP}(\FF_U,\FF'_S)$, $k\in\N$. Since each set
$\iota'_*(\AAA_k)$ is open in
$C^r_{SP}(\FF_U,\FF')$, it is equal to a
countable intersection of open subsets $\BB_{k,\ell}$ of
$C^r_{WP}(\FF_U,\FF')$,
$\ell\in\N$, by Lemma~\ref{l:countable}. Because each $\iota'_*(\AAA_k)$
is also dense in the open subspace
$\iota'_*(C^r(\FF_U,\FF'_S))$ of $C^r_{SP}(\FF_U,\FF')$, and since
$$
\iota^*(\NN^r(U,S))\subset\iota'_*(C^r(\FF_U,\FF'_S))\;,
$$
it follows that each set
$$
\NN^r(U,S)\cap(\iota^*)^{-1}(\iota'_*(\AAA_k))
$$
is dense in $\NN^r_{SP}(U,S)$ by Lemma~\ref{l:iota*}. Therefore
$$
\NN^r(U,S)\cap(\iota^*)^{-1}(\BB_{k,\ell})
$$
is open and dense in $\NN^r_{SP}(U,S)$ by Lemma~\ref{l:composition}-(1).
Moreover
\begin{align*}
\bigcap_{k,\ell=1}^\infty(\iota^*)^{-1}(\BB_{k,\ell})
&=\bigcap_{k=1}^\infty(\iota^*)^{-1}(\iota'_*(\AAA_k))\\
&\subset(\iota^*)^{-1}(\iota'_*(\XX(U,S;A)))\\
&\subset\text{}\pitchfork_{\text{\rm ls}}^r\!\!(\FF,\FF';A)\;.
\end{align*}
So
$$
\NN^r(U,S)\cap\text{}\pitchfork_{\text{\rm ls}}^r\!\!(\FF,\FF';A)
$$
is residual in $\NN^r_{SP}(U,S)$, and the result follows because all possible sets
$\NN^r(U,S)$ form an open covering of $C_{SP}^r(\FF,\FF';A)$.
\end{proof}

\begin{cor}\label{c:ls-transverse}
Assume that $1\le r<\infty$,  and that $A$ is closed in $M'$. If some $\phi\in
C^r(\FF,\FF';A)$ is ls-transverse to $A$ on some neighbourhood of a closed
subset $E\subset M$, then any neighbourhood of $\phi$ in
$C^r_{SP}(\FF,\FF';A)$ contains some foliation map in
$\pitchfork^r_{\text{\rm ls}}\!\!(\FF,\FF';A)$ which equals $\phi$ on some neighbourhood of
$E$.
\end{cor}

\begin{proof}
There are open subsets $U$ and $V$ of $M$ such that $E\subset V$,
$U\cup V=M$, $E\cap\overline{U}=\emptyset$ and
$\phi\pitchfork_{\text{\rm ls},V}A$. By Lemma~\ref{l:H(psi)}, there
is an sp-neighborhood $\MM$ in $C^r_{SP}(\FF|_U,\FF';A)$ of the
restriction $\phi|_U$ such that, for each $\psi\in\MM$, the
combination $H(\psi)$ of $\psi$ and $\phi|_{M\setminus U}$ is a
$C^r$ foliation map $\FF\to\FF'$, and moreover the assignment
$\psi\mapsto H(\psi)$ defines a continuous map $H:\MM_{SP}\to
C^r_{SP}(\FF,\FF')$. Then, given any sp-neighborhood $\NN$ of
$\phi$, there is some neighborhood $\NN'$ of the restriction
$\phi:\FF|_U\to\FF'$ in $\MM_{SP}$ such that $H(\NN')\subset\NN$.
From Corollary~\ref{c:C r SP(cF,cF';A) is Baire} and
Theorem~\ref{t:ls-transverse}, it follows that there is some
$\psi\in\NN'$ which is ls-transverse to $A$. Then $H(\psi)\in\NN$ is
ls-transverse to $A$ and equals $\phi$ on the neighborhood
$M\setminus\overline{U}$ of $E$.
\end{proof}

\section{Leafwise transversality}\label{s:leaf transver}

With the notation of the above section, a point $x\in M$ is said to
be {\em leafwise critical\/} for a $C^r$ foliation map
$\phi:\FF\to\FF'$ if $\phi_*:T_x\FF\to T_{\phi(x)}\FF'$ is not surjective; {\em i.e.\/}, the
leafwise critical points of $\phi$ are the critical points of its
restriction to the leaves. The images by $\phi$ of its leafwise
critical points are called {\em leafwise critical values\/} of
$\phi$, and a point of $M'$ is called a {\em leafwise regular
value\/} of $\phi$ if it is not a leafwise critical value.

Suppose that $A$ is a regular $C^r$ submanifold of $M'$ whose intersection
with
$\FF'$ is clean. A foliation map $\phi\in C^r(\FF,\FF')$ is said to be
{\em leafwise transverse\/} to
$A$ along any subset $K\subset M$, denoted by $\phi\pitchfork_{\text{leaf},K}A$, when
$$
T_{\phi(x)}\FF'=\phi_*(T_x\FF)+T_{\phi(x)}(\FF'|_A)
$$
for all $x\in K\cap\phi^{-1}(A)$; \textit{i.e.}, when its restriction to the leaves is transverse
along $K$ to $\FF'|_A$. If $K$ consists of only one point $x$, then the notation $\phi\pitchfork_{\text{leaf},x}A$ can be used. It is said that $\phi$ is {\em leafwise transverse\/} to $A$, denoted by $\phi\pitchfork_{\text{leaf}}A$, if $\phi\pitchfork_{\text{\rm leaf},M}A$.

Leafwise versions of the Morse-Sard Theorem and the Transversality Theorem
do not hold in general, as shown by the following naive example.

\begin{ex}
Let $\FF$ be the foliation on $\R$
whose leaves are points, and $\FF'$ the
foliation on $\R$ with one leaf. Then
$C^\infty_{SP}(\FF,\FF')=C^\infty_S(\R,\R)$, and the set of leafwise
critical values for any $\phi\in\cinf(\R,\R)$ is its image
$\phi(\R)$. Thus
$\phi$ has no leafwise regular values if it is surjective. Moreover, if $\phi$ is a
diffeomorphism and $\psi$ is close enough to $\phi$ in $C^\infty_S(\R,\R)$, then
$\psi$ is also a diffeomorphism, and thus $\psi$ has no leafwise regular
values either; {\em i.e.\/},
$\psi$ cannot be leafwise transverse to any point.
\end{ex}

Therefore, in general, we consider partial leafwise
transversality. Precisely, for any $\phi\in C^r(\FF,\FF';A)$, let $N=\phi^{-1}(A)$, and
$$
N_0=\{x\in N\ |\ \phi\pitchfork_{\text{\rm leaf},x}A\}\;.
$$

\begin{prop}\label{p:N 0 is open}
$N_0$ is an open subset of $N$.
\end{prop}

\begin{proof}
A point $x\in N$ is in $N_0$ just when the composite
\begin{equation}\label{e:composite on T xFF}
\begin{CD}
T_x\FF @>{\phi_*}>> T_{\phi(x)}\FF' @>>> T_{\phi(x)}\FF'/T_{\phi(x)}(\FF'|_A)
\end{CD}
\end{equation}
is surjective, where the last arrow denotes the canonical projection. Then the result follows because these composites combine to define a vector bundle homomorphism $(T\FF)|_N\to(T\FF')|_A/T(\FF'|_A)$. \end{proof}

\begin{prop}\label{p:N 0 is C r and transverse}
$N_0$ is a regular $C^r$ submanifold of $M$ transverse to $\FF$.
\end{prop}

\begin{proof}
Let $S$ be a good local saturation of $A$. For each $x\in N$, we have $\phi_*(T_xM)\subset T_{\phi(x)}S$, and thus we can consider the composite
\begin{equation}\label{e:composite on T xM}
\begin{CD}
T_xM @>{\phi_*}>> T_{\phi(x)}S @>>> T_{\phi(x)}S/T_{\phi(x)}A\;,
\end{CD}
\end{equation}
where the last arrow denotes the canonical projection. Via the isomorphism~\eqref{e:T xFF/T x(FF| A)}, the restriction of this composition to $T_x\FF$ corresponds to the composition~\eqref{e:composite on T xFF}. So $N_0$ is the set of points $x\in N$ where the restriction of the composition~\eqref{e:composite on T xM} to $T_x\FF$ is surjective. Therefore $\phi$ is ls-transverse to $A$ on $N_0$, and thus $N_0$ is a regular $C^r$ submanifold of $M$ by Proposition~\ref{p:ls-transver}.

For each $x\in N_0$, $T_x N_0$ is the kernel of the composite~\eqref{e:composite on T xM}. Thus that composite induces an isomorphism
\begin{equation}\label{e:T xM/T xN}
T_xM/T_xN\cong T_{\phi(x)}S/T_{\phi(x)}A\;,
\end{equation}
whose restriction to $(T_xN+T_x\FF)/T_xN$ is surjective. Therefore $T_xN+T_x\FF=T_xM$.
\end{proof}

\begin{cor}\label{c:N 0 = transverse}
If $\phi$ is ls-transverse to $A$, then $N$ is a regular $C^\infty$ submanifold of $M$, and
$N_0$ is the set of points where $N$ is transverse to $\FF$.
\end{cor}

\begin{proof}
$N$ is a regular $C^\infty$ submanifold of $M$ by Proposition~\ref{p:ls-transver}. In this case,  the composite~\eqref{e:composite on T xM} induces the isomorphism~\eqref{e:T xM/T xN} for all $x\in N$. Hence $T_xN+T_x\FF=T_xM$ if and only if the restriction of the composite~\eqref{e:composite on T xM} to $T_x\FF$ is surjective, which means that $x\in N_0$.
\end{proof}

\begin{rem}
By using the parametric transversality theorem \cite[pp.~79--80]{Hirsch} and~\eqref{e:T xFF/T x(FF| A)}, it can be proved that, if
$$
r>\max\{0,\dim\FF-\codim_SA\}
$$
for some good local saturation $S$ of $A$, then, for any $\phi$ in $\pitchfork^r_{\text{\rm ls}}\!\!(\FF,\FF';A)$, there is some saturated subset $X\subset M$ of null Lebesgue measure such that $\phi\pitchfork_{\text{\rm leaf},M\setminus X}A$. But we omit the proof because it will not be needed in this paper.
\end{rem}

\section{Approximation and transversality on foliated manifolds with boundary}\label{s:boundary}

With the notation of Section~\ref{s:preliminaries}, the definition of foliation can be adapted to the case where $M$ is a manifold with boundary in the following way (see {\em e.g.\/} \cite{CandelConlon}). For each positive integer $m$, consider the Euclidean half space
$$
H^m=\{(x_1,\dots,x_m)\in\R^m\ |\ x_m\ge0\}\;.
$$
By fixing a homeomorphism between $H^q\times H^p$ and $H^n$, we can assume that the charts of $M$ have values in $H^q\times H^p$. Then the generalization is obtained by considering charts $(U_i,\theta_i)$ of $M$ with values in with $\theta_i(U_i)=T_i\times B_i$, where $T_i$ is an open subset $H^q$ and $B_i$ is a convex open subset of $H^p$. The {\em tangent\/} and {\em transverse boundaries\/} of each $U_i$, respectively denoted by $\partial_\tau U_i$ and $\partial_\pitchfork U_i$, are defined as follows:
\begin{enumerate}

\item If $\partial T_i=\partial B_i=\emptyset$, then $U_i\subset\Int(M)$ and $\partial_\tau U_i=\partial_\pitchfork U_i=\emptyset$.

\item If $\partial T_i\neq\emptyset$ and $\partial B_i=\emptyset$, then $\partial_\tau U_i=\theta_i^{-1}(\partial T_i\times B_i)$ and $\partial_\pitchfork U_i=\emptyset$.

\item If $\partial T_i=\emptyset$ and $\partial B_i\neq\emptyset$, then
$\partial_\pitchfork U_i=\theta_i^{-1}(T_i\times\partial B_i)$ and $\partial_\tau U_i=\emptyset$.

\item If $\partial T_i\neq\emptyset\neq\partial B_i$, then
$$
\partial_\tau U_i=\theta_i^{-1}(\partial T_i\times\Int(B_i))\;,\qquad
\partial_\pitchfork U_i=\theta_i^{-1}(\Int(T_i)\times\partial B_i)\;.
$$
In this case, $\overline{\partial_\tau U_i}\cap\overline{\partial_\pitchfork U_i}$ is the ``corner'' $\theta_i^{-1}(\partial T_i\times\partial B_i)$ of $U$.

\end{enumerate}
Then $\partial_\tau M=\bigcup_i\partial_\tau U_i$ and $\partial_\pitchfork M=\bigcup_i\partial_\pitchfork U_i$ are respectively called the {\em tangent\/} and {\em transverse boundaries\/} of $M$. Observe that $\partial_\tau M$ is a union of leaves, and $\partial_\pitchfork M$ is transverse to the leaves.

All concepts recalled in Section~\ref{s:preliminaries} can be directly extended to the case with boundary. There is only a small subtlety in the definition of a $C^r$ structure of $\FF$ for $r\ge1$ because there is no $C^r$ diffeomorphism between $H^q\times H^p$ and $H^n$. There are two possible solutions to this problem. Either we assume that the case~(4) does not hold, and then we can use a $C^r$ open embedding
$$
(\Int(H^p)\times H^q)\cup(H^p\times\Int(H^q))\to H^n
$$
to consider $C^r$ foliation charts with values in $(\Int(H^p)\times H^q)\cup(H^p\times\Int(H^q))$; in this case, we have $\partial M=\partial_\tau M\cup\partial_\pitchfork M$. Or we assume that $M$ is a $C^r$ manifold with corners: $\overline{\partial_\tau M}\cap\overline{\partial_\pitchfork M}$ is a corner of dimension $n-2$.

Also, the concepts of $C^r$ foliation map, the weak and strong plaquewise topologies, and ls-transversality can be directly extended to the foliated manifolds with boundary. In particular, the sets $C^r(\FF,\FF';A)$ and $\pitchfork^r_{\text{\rm ls}}\!\!(\FF,\FF';A)$ can be defined for the case with boundary in exactly the same way as they were introduced in Sections~\ref{s:clean} and~\ref{s:ls-transver}. These direct extensions of basic definitions will be used in the next sections.

Some of our main results can be also extended to the case with boundary. For instance, there are versions of Theorem~\ref{t:approximation} and Proposition~\ref{p:ls-transver} for foliated manifolds with boundary, which are similar to the corresponding results for manifolds with boundary (Theorem~3.3 of Chapter~2 and Theorem~4.2 of Chapter~1 in \cite{Hirsch}). These extended results will not be used for the sake of simplicity.

\section{Leafwise homotopies}\label{s:homotopies}

For $0\le r\le\infty$, let $\FF$ and $\FF'$ be $C^r$ foliations on $C^r$ manifolds $M$ and $M'$.
Recall that an {\em integrable\/} or {\em leafwise homotopy\/} between foliation maps
$\phi,\psi:\FF\to\FF'$ is a homotopy
$H:M\times I\to M'$ between $\phi$ and $\psi$ so that,  for each $x\in M$, the path $I\to M'$
given by $t\mapsto H(x,t)$ lies in a leaf of $\FF'$; \textit{i.e.}, the homotopy $H$ is a foliation
map $\FF\times I\to\FF'$, where $\FF\times I$ is the foliation
with leaves $L\times I$ for leaves $L$ of $\FF$. Observe that $\partial(M\times I)=\partial_\pitchfork(M\times I)$.

\begin{prop}\label{p:C r integrable homotopy}
Let $H$ be an integrable homotopy between $C^r$ foliation maps $\phi,\psi:\FF\to\FF'$.
Then any neighborhood of $H$ in $C_{SP}(\FF\times I,\FF')$ contains some $C^r$ integrable homotopy between $\phi$ and
$\psi$.
\end{prop}

\begin{proof}
Let $\overline{H}:\FF\times\R\to\FF'$ be the foliation map defined by
$$
\overline{H}(x,t)=
\begin{cases}
H(x,t)&\text{if $t\in I$}\\
\phi(x)&\text{if $t\le0$}\\
\psi(x)&\text{if $t\ge1$.}\\
\end{cases}
$$
This $\overline{H}$ is $C^{r,0}$, and it is $C^r$ on $M\times(\R\setminus
I)$. By Theorem~\ref{t:approximation}, for any $\epsilon>0$, any sp-neighborhood of $\overline{H}$ contains a $C^r$ foliation map $F:\FF\times\R\to\FF'$ which equals
$\overline{H}$ on $M\times(-\infty,-\epsilon]$ and $M\times[1+\epsilon,\infty)$.
Thus the foliation map $G:\FF\times I\to\FF'$, defined by
$G(x,t)=F(x,(1+2\epsilon)t-\epsilon)$, is a $C^r$ integrable homotopy between $\phi$ and
$\psi$. Moreover $G$ is as sp-close to $H$ as desired when $G$ is sp-close enough to $\overline{H}$ and $\epsilon$ is small enough.
\end{proof}

\begin{prop}\label{p:integrable homotopy}
With the above notation, any foliation map $\phi:\FF\to\FF'$ has a neighborhood $\NN$ in $C_{SP}(\FF,\FF')$ such that there is a continuous map $\NN_{SP}\to C_{SP}(\FF\times I,\FF')$, $\psi\mapsto H_\psi$, so that each $H_\psi$ is an integrable homotopy between $\phi$ and $\psi$, and $H_\phi(x,t)=\phi(x)$ for all $x\in M$ and $t\in I$.
\end{prop}

\begin{proof}
Consider the following data:
\begin{itemize}

\item a locally finite family $\{U_i\}$ of open subsets of $M$;

\item a covering of $M$ by compact subset, $\{K_i\}$, with $K_i\subset U_i$ for
all $i$; and,

\item for each $i$, a foliation chart $\theta'_i:U_i'\to T'_i\times\R^{p'}$ of $\FF'$,
$p'=\dim\FF'$, such that $\phi(K_i)\subset U'_i$.

\end{itemize}
Then take the sp-neighbourhood $\NN=\NN(\phi,\Theta,\Theta',\KK)$ of $\phi$ defined in
Section~\ref{sec:foliation maps}.

Since $\{U_i\}$ is locally finite, $i$ runs in a countable index set; say, either in $\{1,\dots,n\}$ for some $n\in\Z_+$, or in
$\Z_+$. Set $\phi_0=\phi$.

\begin{claim}\label{cl:Cr integrable homotopy}
For each $i$, there is a foliation map $\phi_i:\FF\to\FF'$, and there is an integrable homotopy $H_i:\FF\times I\to\FF'$ between $\phi_{i-1}$ and $\phi_i$ if $i>0$, such that:
\begin{itemize}

\item $\phi_i(U_j)\subset U'_j$ for all $j$;

\item $\phi_i(x)$ and $\phi_{i-1}(x)$ lie in the same plaque of $\theta'_i$ for all $x\in U_i$ and
$i$;

\item $\phi_i(x)=\psi(x)$ for all $x\in K_1\cup\dots\cup K_i$ for all $i$;

\item $\phi_i(x)=\phi_{i-1}(x)$ for all $x\in M\setminus U_i$ and $i$;

\item $\phi_i$ and $H_i$ depend sp-continuously on $\phi_{i-1}$ for $i>0$.

\end{itemize}
\end{claim}

This assertion is proved by induction on $i$.  We
have already set $\phi_0=\phi$. Now suppose that $\phi_{i-1}$ is given for some $i>0$. Fix a continuous map $\lambda_i:M\to I$ such that $\supp\lambda_i\subset U_i$ and
$\lambda_i(x)=1$ for all $x\in K_i$. Then it is easy to show that the statement of
Claim~\ref{cl:Cr integrable homotopy} holds with the following definitions of $\phi_i$ and $H_i$.
Let $H_i(x,t)=\phi_{i-1}(x)$ for $x\in M\setminus U_i$,
$$
\theta'_i\circ H_i(x,t)=
(1-t\,\lambda_i(x))\cdot\theta'_i\circ\phi_{i-1}(x)+
t\,\lambda_i(x)\cdot\theta'_i\circ\psi(x)
$$
for $x\in U_i$, and $\phi_i(x)=H_i(x,1)$ for any $x\in M$.

The result follows from Claim~\ref{cl:Cr integrable homotopy} in the following way. Take
a partition
$$
0=t_1<\dots<t_n<t_{n+1}=1
$$
of $I$ when $i$ runs in $\{1,\dots,n\}$, and take any sequence
$t_i\uparrow1$ with $t_1=0$ when $i$ runs in $\Z_+$. In any case, let
$\rho_i:[t_i,t_{i+1}]\to I$ be any orientation preserving
homeomorphism for each $i$. Then an integrable homotopy $H_\psi:\FF\times I\to\FF'$ between $\phi$ and $\psi$ can be defined by
$$
H_\psi(x,t)=
\begin{cases}
H_i(x,\rho_i(t))&\text{if $t\in[t_i,t_{i+1}]$ for some $i$,}\\
\psi(x)&\text{if $t=1$.}
\end{cases}
$$
Moreover it is easy to check that $H_\psi$ depends continuously on $\psi$.
\end{proof}

\begin{rem}
When the order of differentiability $r$ is high enough,
Proposition~\ref{p:integrable homotopy} could be proved by using more descriptive tools.
For instance, a riemannian metric on the leaves, $C^r$ on $M$, could
be used as follows. If $\phi$ and $\psi$ are close enough in
$C_{SP}(\FF,\FF')$, for each $x\in M$, there is a unique
minimal geodesic path $c_x$ in a leaf of $\FF'$ from $\phi(x)$ to
$\psi(x)$, which is defined in the unit interval $I$ and depends
smoothly on $x$. Therefore an integrable homotopy $H$ between $\phi$
and $\psi$ is defined by $H(x,t)=c_x(t)$.
\end{rem}

Let $A$ be a regular $C^r$ submanifold of $M'$ whose intersection with $\FF'$ is clean.

\begin{cor}\label{c:integrable homotopy}
Each $\phi\in C^r(\FF,\FF';A)$ has some neighborhood $\NN$ in $C^r_{SP}(\FF,\FF';A)$ such that there exists an integrable homotopy in $C^r(\FF\times I,\FF';A)$ between $\phi$ and any map in $\NN$.
\end{cor}

\begin{proof}
This is a straightforward consequence of Propositions~\ref{p:C
r(cF,cF';A) is sp-open},~\ref{p:C r integrable homotopy}
and~\ref{p:integrable homotopy}.
\end{proof}

\begin{cor}\label{c:integrable homotopy to pitchfork ls}
Suppose that $r\ge1$. If $A$ is closed in $M'$, then there exists an integrable homotopy in $C^{r,0}(\FF\times I,\FF';A)$ between each foliation map in $C^{r,0}(\FF,\FF';A)$ and some map in $\pitchfork^r_{\text{\rm ls}}\!\!(\FF,\FF';A)$.
\end{cor}

\begin{proof}
This follows directly from Corollaries~\ref{c:C r SP(cF,cF';A) is Baire} and~\ref{c:C r(cF,cF';A) is dense in C r,s(cF,cF';A)}, Theorem~\ref{t:ls-transverse} and Proposition~\ref{p:integrable homotopy}.
\end{proof}

\begin{prop}\label{p:ls-transverse homotopy}
Suppose that $1\le r<\infty$ and $A$ is closed in $M'$. If there is an integrable homotopy in $C(\FF\times I,\FF';A)$ between two foliation maps $\phi$ and $\psi$ in $\pitchfork^r_{\text{\rm ls}}\!\!(\FF,\FF';A)$, then there exists an integrable homotopy $G$ in $\pitchfork^r_{\text{\rm ls}}\!\!(\FF\times I,\FF';A)$ such that $G(x,t)=\phi(x)$ on some neighborhood of $M\times\{0\}$ in $M\times I$, and $G(x,t)=\psi(x)$ on some neighborhood of $M\times\{1\}$ in $M\times I$.
\end{prop}

\begin{proof}
Given an integrable homotopy $H\in C(\FF\times I,\FF';A)$ between $\phi$ and $\psi$, let $\overline{H}\in C^{r,0}(\FF\times\R,\FF';A)$ be defined by
$$
\overline{H}(x,t)=
\begin{cases}
H(x,t) & \text{if $t\in I$}\\
\phi(x) & \text{if $t<0$}\\
\psi(x) & \text{if $t>1$\;.}
\end{cases}
$$
Since $\overline{H}$ is $C^r$ on $M\setminus I$, by
Corollary~\ref{c:C r(cF,cF';A) is dense in C r,s(cF,cF';A)}, there
is a foliation map $F\in C^r(\FF\times\R,\FF';A)$ which equals
$\overline{H}$ on $M\times(\R\setminus[-1,2])$. By
Corollary~\ref{c:ls-transverse}, there exists some foliation map
$F'$ in $\pitchfork^r_{\text{\rm ls}}\!\!(\FF\times\R,\FF';A)$ which
equals $F$ on $\R\setminus[-2,3]$. Then the result follows with
$G:\FF\times I\to\FF'$ defined by $G(x,t)=F'(x,7t-3)$.
\end{proof}

\section{Coincidence points}\label{s:coincidence}

A  specially useful case of ls-transversality is studied in this section.

For $0\le r\le\infty$, let $\FF$ be a $C^r$ foliation on a manifold
$M$. Then the diagonal $\Delta_M$ of $M\times M$ is not
transverse to $\FF\times\FF$, but we have the following.

\begin{prop}\label{p:Delta M cap(FF times FF) clean}
The intersection of $\Delta_M$ and $\FF\times\FF$ is clean.
\end{prop}

\begin{proof}
For any foliation chart
$\theta=(\theta_1,\theta_2) : U\to T\times B$ of $\FF$, the foliation chart
$$
\sigma=(\theta_1\times\theta_1,\theta_2\times\theta_2): U\times U
\to T\times T\times B\times B
$$
of $\FF\times\FF$ satisfies
$$
\sigma((U\times U)\cap\Delta_M)=\Delta_T\times\Delta_B\;,
$$
where $\Delta_T$ and $\Delta_B$
are the diagonals of $T\times T$ and $B\times B$.
\end{proof}

\begin{cor}\label{c:S}
For a regular covering $\{U_i\}$ of $M$ by distinguished open sets for $\FF$, the set
$$
S=\{(x,y)\in M\times M\ |\ \text{$x$ and $y$ are in the same plaque of some $U_i$}\}
$$
is a good local saturation of $\D_M$.
\end{cor}

\begin{proof}
This is a direct consequence of the proofs of Propositions~\ref{p:Delta M cap(FF times FF) clean} and~\ref{p:clean}.
\end{proof}

Let $\FF'$ be another $C^r$ foliation on a manifolds $M'$. Take a regular covering $\{U'_i\}$ of $M'$ by distinguished open sets of $\FF$, and let $S'$ be the corresponding good local saturation of $\D_{M'}$ given by Corollary~\ref{c:S}. For each $x'\in M'$, the leaf of $(\FF'\times\FF')|_{\D_{M'}}$ through $(x',x')$ is $\D_{L'}$, where $L'$ is the leaf of $\FF'$ through $x'$.

\begin{prop}\label{p:(phi,psi) in C r(cF,cF' times cF';D M)}
If $\phi,\psi\in C^r(\FF,\FF')$ induce the same map $M/\FF\to M'/\FF'$, then $(\phi,\psi)\in C^r(\FF,\FF'\times\FF';\D_{M'})$.
\end{prop}

\begin{proof}
Any $x\in(\phi,\psi)^{-1}(\D_{M'})$ has some distinguished neighborhood $U_x$ in $M$ such that, for each plaque $P$ of $U_x$, there is a plaque $P'$ of some $U'_i$ such that $\phi(P),\psi(P)\subset P'$. Hence $(\phi,\psi)(U_x)\subset S'$ for all $x\in\Coin(\phi,\psi)$, yielding the statement.
\end{proof}

The following result has an analogous proof.

\begin{prop}\label{p:homotopy in C r(cF times I,cF' times cF';D M)}
For foliation maps $\phi,\psi,\xi,\zeta\in C^r(\FF,\FF')$ inducing the same map $M/\FF\to M'/\FF'$, any $C^r$ integrable homotopy between $(\phi,\psi)$ and $(\xi,\zeta)$ is in $C^r(\FF\times I,\FF'\times\FF';\D_{M'})$.
\end{prop}

From now on in this section, assume that $r\ge1$.  In this case,~\eqref{e:T xFF/T x(FF| A)} becomes
\begin{equation}\label{e:T (x',x')(L' times L')/T (x',x')D L'}
T_{(x',x')}(L'\times L')/T_{(x',x')}\D_{L'}\cong T_{(x',x')}S'/T_{(x',x')}\D_{M'}\;.
\end{equation}
Furthermore
\begin{align*}
T_{(x',x')}(L'\times L')&\equiv T_{x'}L'\oplus T_{x'}L'\;,\\
T_{(x',x')}\D_{L'}&\equiv\{(v,v)\ |\ v\in T_{x'}L'\}\;,
\end{align*}
yielding a canonical isomorphism
\begin{equation}\label{e:T x'L'}
T_{(x',x')}(L'\times L')/T_{(x',x')}\D_{L'}\cong T_{x'}L'\;,\qquad[v,w]\mapsto v-w\;,
\end{equation}
where $[v,w]$ is the normal vector represented by $(u,v)\in T_{(x',x')}(L'\times L')$.

Take $\phi,\psi\in C^r(\FF,\FF')$ inducing the same map $M/\FF\to M'/\FF'$.  A point $x\in M$ is said to be of {\em coincidence\/} of $\phi$ and $\psi$ if $\phi(x)=\psi(x)$. Let $\Coin(\phi,\psi)$ denote the set of coincidence points of $\phi$ and $\psi$; {\em i.e.\/}, $\Coin(\phi,\psi)=(\phi,\psi)^{-1}(\D_{M'})$, which is closed in $M$. A coincidence point $x$ of $\phi$ and $\psi$ is called {\em leafwise simple\/} when $\phi_*-\psi_*:T_x\FF\to T_{x'}\FF'$ is onto, where $x'=\phi(x)$. The set of leafwise simple coincidence
points of $\phi$ and $\psi$ will be denoted by $\Coin_0(\phi,\psi)$.

\begin{prop}\label{p:leafwise simple coincidence point}
A point $x\in\Coin(\phi,\psi)$ is leafwise simple if and only if $(\phi,\psi)$ is leafwise transverse to $\D_M$ at $x$
\end{prop}

\begin{proof}
Let $L$ be the leaf of $\FF$ through $x$, and let $L'$ be the leaf of $\FF'$ through $x'=\phi(x)$. In this case,~\eqref{e:composite on T xFF} becomes
$$
\begin{CD}
T_xL @>{(\phi,\psi)_*}>> T_{(x',x')}(L'\times L') @>>> T_{(x',x')}(L'\times L')/T_{(x',x')}\D_{L'}\;.
\end{CD}
$$
Via~\eqref{e:T x'L'}, this composite corresponds to $\phi_*-\psi_*:T_xL\to T_{x'}L'$, and the result follows.
\end{proof}

By Proposition~\ref{p:leafwise simple coincidence point}, the following are particular cases of Propositions~\ref{p:N 0 is open},~\ref{p:N 0 is C r and transverse} and~\ref{c:N 0 = transverse}.

\begin{cor}\label{c:Coin 0 is open}
$\Coin_0(\phi,\psi)$ is open in $\Coin(\phi,\psi)$.
\end{cor}

\begin{cor}\label{c:Coin 0 is C r and transverse}
$\Coin_0(\phi,\psi)$ is a regular $C^r$ submanifold of $M$ transverse to $\FF$.
\end{cor}

\begin{cor}\label{c:Coin 0 = transverse}
If $(\phi,\psi)$ is ls-transverse to $\D_{M'}$, then $\Coin(\phi,\psi)$ is a regular submanifold of $M$, and $\Coin_0(\phi,\psi)$ is the set of points where $\Coin(\phi,\psi)$ is transverse to $\FF$.
\end{cor}

Now~\eqref{e:T xM/T xN} becomes
\begin{equation}\label{e:T xM/T xCoin 0}
T_xM/T_x\Coin_0(\phi,\psi)\cong T_{(x',x')}S'/T_{(x',x')}\D_{M'}
\end{equation}
for all $x\in\Coin_0(\phi,\psi)$, where $x'=\phi(x)$. If moreover $(\phi,\psi)$ is ls-transverse to $\D_{M'}$, then~\eqref{e:T xM/T xN} becomes
\begin{equation}\label{e:T xM/T xCoin}
T_xM/T_x\Coin(\phi,\psi)\cong T_{(x',x')}S'/T_{(x',x')}\D_{M'}
\end{equation}
for all $x\in\Coin(\phi,\psi)$, where $x'=\phi(x)$.

\begin{cor}\label{c:Coin 0 transversal}
If $\dim\FF=\dim\FF'$, then $\Coin_0(\phi,\psi)$ is a $C^r$ transversal of $\FF$.
\end{cor}

\begin{proof}
This follows from Corollary~\ref{c:Coin 0 is C r and transverse} because the codimension of $\Coin_0(\phi,\psi)$ is equal to $\dim\FF'$ by~\eqref{e:T xM/T xCoin},~\eqref{e:T (x',x')(L' times L')/T (x',x')D L'} and~\eqref{e:T x'L'}.
\end{proof}

\begin{rem}
It is well known that the existence of a closed transversal is a strong
restriction for a foliation. So, by Corollary~\ref{c:Coin 0 transversal}, if $\dim\FF=\dim\FF'$, then there may not be any
$\phi,\psi\in C^r(\FF,\FF)$ inducing the same map $M/\FF\to M/\FF'$ and with $\Coin(\phi,\psi)=\Coin_0(\phi,\psi)$.
\end{rem}

Let $\nu_0$ denote the normal bundle of $\Coin_0(\phi,\psi)$, and $\nu$ the normal bundle of $\Coin(\phi,\psi)$ if $(\phi,\psi)$ is ls-transverse to $\D_{M'}$. The proof of the following result is elementary.

\begin{prop}\label{p:hom}
The isomorphisms~\eqref{e:T xM/T xCoin 0},~\eqref{e:T (x',x')(L' times L')/T (x',x')D L'} and~\eqref{e:T x'L'} define a $C^r$ homomorphism $\nu_0\to T\FF'$.  If $(\phi,\psi)$ is ls-transverse to $\FF'$, then~\eqref{e:T xM/T xCoin},~\eqref{e:T (x',x')(L' times L')/T (x',x')D L'} and~\eqref{e:T x'L'} define a $C^r$ homomorphism $\nu\to T\FF'$, which extends the above homomorphism $\nu_0\to T\FF'$.
\end{prop}

Suppose that $\dim\FF=\dim\FF'$. Then, by Corollary~\ref{c:Coin 0 transversal}, the inclusion map $T\FF\hookrightarrow TM$ induces a canonical $C^r$ isomorphism
\begin{equation}\label{e:nu 0}
(T\FF)|_{\Coin_0(\phi,\psi)}\cong\nu_0\;.
\end{equation}
The following result follows from the proof of Proposition~\ref{p:leafwise simple coincidence point}.

\begin{prop}\label{p:phi *-psi *}
Via~\eqref{e:nu 0}, the homomorphism $\nu_0\to T\FF'$ of Proposition~\ref{p:hom} corresponds to
$$
\phi_*-\psi_*:(T\FF)|_{\Coin_0(\phi,\psi)}\to T\FF'\;.
$$
\end{prop}

\section{Fixed points}\label{s:fixed}

For each $\phi\in C^r_{\text{\rm leaf}}(\FF,\FF)$, its graph map $\Graph(\phi)=(\id_M,\phi):M\to M\times M$ is in $C^r(\FF,\FF\times\FF;\D_M)$ by Proposition~\ref{p:(phi,psi) in C r(cF,cF' times cF';D M)}, and $\Graph(\phi)^{-1}(\D_M)=\Coin(\id_M,\phi)$ is the set of fixed points of $\phi$, denoted by $\Fix(\phi)$. By Lemma~\ref{l:product}, the assignment $\phi\mapsto\Graph(\phi)$ defines a continuous map
$$
\Graph:C^r_{\text{\rm leaf},SP}(\FF,\FF)\to C^r_{SP}(\FF,\FF\times\FF;\D_M)\;.
$$

Suppose that $r\ge 1$. It may not be possible to approach a $C^r$ leaf preserving foliation map $\FF\to\FF$ by other ones whose graphs are ls-transverse to $\D_M$; however, this is possible for leaf preserving foliation diffeomorphisms according to the following.

\begin{thm}\label{t:Diff r(cF) cap Graph -1(pitchfork r ls(cF,cF times cF;D M))}
With the above notation, the set
$$
\Diff^r(\FF)\cap\Graph^{-1}(\pitchfork^r_{\text{\rm ls}}\!\!(\FF,\FF\times\FF;\D_M))
$$
is residual in $\Diff^r_{SP}(\FF)$.
\end{thm}

\begin{proof}
Via the canonical identity given by Lemma~\ref{l:product}, we have
$$
\Diff^r(\FF)\times\Diff^r(\FF)\subset C^r(\FF,\FF\times\FF;\D_M)
$$
by Proposition~\ref{p:(phi,psi) in C r(cF,cF' times cF';D M)}. By Corollary~\ref{c:Diff is open} and Theorem~\ref{t:ls-transverse}, the
set
$$
\AAA=(\Diff^r(\FF)\times\Diff^r(\FF))\cap
\text{}\pitchfork_{\text{\rm ls}}^r\!\!(\FF,\FF\times\FF;\D_M)
$$
is residual in $\Diff^r_{SP}(\FF)\times\Diff^r_{SP}(\FF)$. But any
$(\phi,\psi)\in\Diff^r(\FF)\times\Diff^r(\FF)$ is equal to
the composite
\begin{equation}\label{e:(phi,psi)}
\begin{CD}
\FF @>{\phi}>> \FF
@>{\Graph(\psi\circ\phi^{-1})}>> \FF\times\FF\;,
\end{CD}
\end{equation}
which is ls-transverse to $\D_M$ if and only if so is
$\Graph(\psi\circ\phi^{-1})$. Therefore
$\Graph(\psi\circ\phi^{-1})$ is ls-transverse to $\D_M$ for a residual set of pairs
$(\phi,\psi)$ in $\Diff^r_{SP}(\FF)\times\Diff^r_{SP}(\FF)$.
Then the result follows because the map
$$
\Diff^r_{SP}(\FF)\times\Diff^r_{SP}(\FF)\to\Diff^r_{SP}(\FF)\;,\quad
(\phi,\phi)\mapsto\psi\circ\phi^{-1}\;,
$$
is continuous and open by Lemma~\ref{l:Diff is a topological group}.
\end{proof}

\begin{cor}\label{c:Diff r(cF) cap Graph -1(pitchfork r ls(cF,cF times cF;D M))}
The set
$$
\Diff^r(\FF)\cap\Graph^{-1}(\pitchfork^r_{\text{\rm ls}}\!\!(\FF,\FF\times\FF;\D_M))
$$
is dense in $\Diff^r_{SP}(\FF)$.
\end{cor}

\begin{proof}
This follows from  Theorem~\ref{t:Diff r(cF) cap Graph -1(pitchfork r ls(cF,cF times cF;D M))} and Corollary~\ref{c:Diff is Baire}.
\end{proof}

A fixed point $x$ of a foliation map $\phi\in C^r(\FF,\FF)$ is called {\em leafwise simple\/} when $\id-\phi_*:T_x\FF\to T_x\FF$ is an isomorphism; {\em i.e.\/}, when $x$ is a leafwise simple coincidence point of $\id_M$ and $\phi$, which means that $\Graph(\phi):\FF\to\FF\times\FF$ is leafwise transverse to $\D_M$ at $x$ by Proposition~\ref{p:leafwise simple coincidence point}. The set of leafwise simple fixed points of $\phi$ will be denoted by $\Fix_0(\phi)$. The following are particular cases of Corollaries~\ref{c:Coin 0 is open}--\ref{c:Coin 0 transversal}.

\begin{cor}\label{c:Fix 0 is open}
$\Fix_0(\phi)$ is open in $\Fix(\phi)$.
\end{cor}

\begin{cor}\label{c:Fix 0 transversal}
$\Fix_0(\phi)$ is a $C^r$ transversal of $\FF$.
\end{cor}

\begin{cor}\label{c:Fix 0 = transverse}
If $\Graph(\phi)$ is ls-transverse to $\D_M$, then $\Fix(\phi)$ is a regular submanifold of $M$, and $\Fix_0(\phi)$ is the set of points where $\Fix(\phi)$ is transverse to $\FF$.
\end{cor}

\section{$\Lambda$-coincidence and $\Lambda$-Lefschetz number}\label{s:Lambda-coincidence}

Consider the notation of Section~\ref{s:coincidence}. Suppose that $M$ is closed and $r\ge1$. For the sake of simplicity, assume also that $M=M'$ and $\FF=\FF'$. Let $p=\dim\FF$ and $q=\codim\FF$.

For $\phi,\psi\in C^r(\FF,\FF)$ inducing the same map $M/\FF\to M/\FF$, define the locally constant map $\epsilon_{\phi,\psi}:\Coin_0(\phi,\psi)\to\{\pm1\}$ by
$$
\epsilon_{\phi,\psi}(x)=\sign\det(\phi_*-\psi_*:T_x\FF\to T_x\FF)\;.
$$
By Corollaries~\ref{c:Coin 0 = transverse} and~\ref{c:Coin 0 transversal}, if $(\phi,\psi)\pitchfork_{\text{\rm ls}}\D_M$, then $\Coin(\phi,\psi)$ is a regular $C^r$ submanifold of $M$ of dimension $q$, and $\Coin_0(\phi,\psi)$ is the subset of points where $\Coin(\phi,\psi)$ is transverse to the leaves.

Let $\Lambda$ be a transverse invariant measure of $\FF$ absolutely continuous with respect to the Lebesgue measure; {\em i.e.\/}, $\Lambda$ can be considered as a continuous density on any $C^r$ transversal of $\FF$. In particular, we can consider $\Lambda$ as a continuous density on $\Coin_0(\phi,\psi)$.

\begin{thm}\label{t:int}
With the above notation, we have the following:
\begin{enumerate}

\item If $(\phi,\psi)\pitchfork_{\text{\rm ls}}\D_M$, then the integral
$$
\int_{\Coin_0(\phi,\psi)}\epsilon_{\phi,\psi}\,\Lambda
$$
is defined and finite.

\item Let $\xi,\zeta\in C^r(\FF,\FF)$ inducing the same map $M/\FF\to M/\FF$, and such that $(\xi,\zeta)\pitchfork_{\text{\rm ls}}\D_M$. If there exists an integrable homotopy between $(\phi,\psi)$ and $(\xi,\zeta)$, then
$$
\int_{\Coin_0(\phi,\psi)}\epsilon_{\phi,\psi}\,\Lambda
=\int_{\Coin_0(\xi,\zeta)}\epsilon_{\xi,\zeta}\,\Lambda\;.
$$

\end{enumerate}
\end{thm}

\begin{proof}
By taking appropriate lifts to finite coverings of tangential and transverse orientations, we can assume that $\FF$ is tangentially and transversely oriented. So $M$ is also oriented so that, if $\omega$ and $\chi$ are $C^r$ sections of $\bigwedge^q(TM/T\FF)^*$ and $\bigwedge^pT\FF^*$ representing the transverse and tangential orientations of $\FF$, respectively, then $\omega\wedge\chi$ is a $C^r$ differential $(p+q)$-form representing the orientation of $M$.

Let $P=\Coin(\phi,\psi)$ and $P_0=\Coin_0(\phi,\psi)$, and let $\nu$ and $\nu_0$ denote the corresponding normal bundles. The orientation of $\FF$ induces a transverse orientation of $P$ so that the homomorphism $\nu\to T\FF$, given by Proposition~\ref{p:hom}, is fiberwise orientation preserving. Let $\sigma$ be a $C^r$ section of $\bigwedge^p\nu^*$ representing this transverse orientation. It induces an orientation of $P$, represented by a $C^r$ differential $q$-form $\tau$, so that $\sigma\wedge\tau$ represents the orientation of $M$ on $P$.

By Proposition~\ref{p:phi *-psi *}, the transverse orientation represented by $\sigma$ on $P_0$ corresponds to the orientation of $(T\FF)|_{P_0}$ represented by $\epsilon_{\phi,\psi}\,\chi$. It follows that $\epsilon_{\phi,\psi}\,\omega$ represents the same orientation as $\tau$ on $P_0$.

Since $\FF$ is transversely oriented, $\Lambda$ can be considered as a continuous basic $q$-form $\alpha$. The restrictions of $\alpha$ and $\omega$ to the $C^r$ transversal $P_0$ are also denoted by $\alpha$ and $\omega$. Thus, on $P_0$, $\Lambda$ can be identified to $\alpha$ if we consider the orientation of $P_0$ represented by $\omega$, and $\Lambda$ can be identified to $\epsilon_{\phi,\psi}\,\alpha$ if we consider the orientation of $P_0$ represented by $\tau$. Therefore, by considering the orientation represented by $\tau$, we get
$$
\int_{P_0}\epsilon_{\phi,\psi}\,\Lambda
=\int_{P_0}\alpha
=\int_P\alpha\;,
$$
which is defined and finite because $M$ is compact and $P$ is closed. This completes the proof of~(1).

Now, let $Q=\Coin(\xi,\zeta)$ and $Q_0=\Coin_0(\xi,\zeta)$, and let $\nu'$ and $\nu'_0$ denote the corresponding normal bundles. As above, the orientation of $\FF$ induces a transverse orientation of $Q$. Let $\sigma'$ be a $C^r$ section of $\bigwedge^p\nu^{\prime*}$ representing this transverse orientation, and let $\tau'$ be a $C^r$ differential $q$-form on $Q$ representing the induced orientation of $Q$ as above. Again, $\epsilon_{\xi,\zeta}\,\omega$ represents the same orientation on $Q_0$ as $\tau'$. Moreover, on $Q_0$, $\Lambda$ can be identified to $\epsilon_{\xi,\zeta}\,\alpha$ if we consider the orientation of $Q_0$ represented by $\tau'$.

Assume that $(\phi,\psi)$ and $(\psi,\zeta)$ are integrably homotopic. By Propositions~\ref{p:ls-transverse homotopy} and~\ref{p:homotopy in C r(cF times I,cF' times cF';D M)}, there is an integrable homotopy $G$ in $\pitchfork^1_{\text{\rm ls}}\!\!(\FF\times I,\FF\times\FF;\D_M)$ such that $G(x,t)=(\phi(x),\psi(x))$ for all $(x,t)$ in some neighborhood of $M\times\{0\}$, and $G(x,t)=(\xi(x),\zeta(x))$ for all $(x,t)$ in some neighborhood of $M\times\{1\}$. Then, by applying Proposition~\ref{p:ls-transver} to the restriction $G:\FF\times(0,1)\to\FF\times\FF$, it follows that $N=G^{-1}(\D_M)$ is a regular $C^1$ submanifold of $M\times I$ of dimension $q+1$ such that
\begin{gather}
\partial N=(P\times\{0\})\cup(Q\times\{1\})\;,\label{e:partial N}\\
N\cap(M\times([0,\epsilon)\cup(1-\epsilon,1]))=(P\times[0,\epsilon))\cup(Q\times(1-\epsilon,1])
\label{e:N cap(M times([0,epsilon)cup(1-epsilon,1]))}
\end{gather}
for some $\epsilon>0$.

Let $\nu''$ denote the normal bundles of $N$. By~\eqref{e:N cap(M times([0,epsilon)cup(1-epsilon,1]))}, the restrictions of $\nu''$ to $P\times[0,\epsilon)$  and $Q\times(1-\epsilon,1]$ can be canonically identified to the pull-backs of $\nu$ and $\nu'$ by the first factor projections. Via these identifications, the restrictions to $P\times[0,\epsilon)$  and $Q\times(1-\epsilon,1]$ of the $C^1$ homomorphism $\nu''\to T\FF$, given Proposition~\ref{p:hom}, is defined by the above homomorphism $\nu\to T\FF$ and $\nu'\to T\FF$.

Let $dt$ denote the standard volume form of $\R$. The pull-backs of $\omega$, $\chi$, $\sigma$, $\tau$, $\sigma'$, $\tau'$ and $dt$ to $M\times I$ by the factors projections are denoted with the same symbols. Consider the orientation of $M\times I$ represented by the top degree $C^r$ differential form $\omega\wedge\chi\wedge dt$.  As above, the orientation of $\FF$ induces a transverse orientation of $N$. Let $\sigma''$ be a $C^1$ section of $\bigwedge^p\nu^{\prime\prime*}$ representing this transverse orientation. These orientations of $\nu''$ and $M$ induce an orientation of $N$ as above, which is represented by a $C^1$ differential $(q+1)$-form $\tau''$ on $N$.

By the above identities of the restrictions of the homomorphism $\nu''\to T\FF$ to $P\times[0,\epsilon)$  and $Q\times(1-\epsilon,1]$, we can choose $\sigma''$ so that its restrictions to $P\times[0,\epsilon)$  and $Q\times(1-\epsilon,1]$ are $\sigma$ and $\sigma'$, respectively. So the restrictions of $\tau''$ to $P\times[0,\epsilon)$  and $Q\times(1-\epsilon,1]$ are $\tau\wedge dt$ and $\tau'\wedge dt$, respectively. Therefore, by~\eqref{e:partial N}, the induced orientation on $\partial N$ is represented by $\tau$ on $P\times\{0\}$ and $-\tau'$ on $Q\times\{1\}$.

The pull-back of $\alpha$ to $M\times I$, as well as its restriction to $N$, is also denoted by $\alpha$. Since $\alpha$ is a continuous basic $q$-form of $\FF$, its pull-back to $M\times I$, also denoted by $\alpha$, is a continuous basic $q$-form of $\FF\times I$. Therefore, there is a sequence of $C^1$ horizontal $q$-forms $\alpha_i$ on $M\times I$ converging to $\alpha$ such that $d\alpha_i$ converges to zero, where uniform convergence is considered because $M\times I$ is compact. So, by the Stokes formula,
\begin{multline*}
\int_{P_0}\epsilon_{\phi,\psi}\,\Lambda-\int_{Q_0}\epsilon_{\xi,\zeta}\,\Lambda
=\int_P\alpha-\int_Q\alpha\\
=\int_{\partial N}\alpha
=\lim_i\int_{\partial N}\alpha_i
=\lim_i\int_Nd\alpha_i
=0\;,
\end{multline*}
completing the proof of~(2).
\end{proof}

\begin{rem}
Let $T=P_0\sqcup Q_0$, and let $\epsilon:T\to\{\pm1\}$ denote the combination of $\epsilon_{\phi,\psi}$ and $\epsilon_{\xi,\zeta}$. If the leaves of $(\FF\times I)|_{N_0}$ were compact, then the non-closed ones form a $C^1$ family of paths joining points of $T$. Thus they define a diffeomorphism $h:T\to T$ which is combination of holonomy transformations of $\FF$ and satisfies $\epsilon\circ h=-\epsilon$. Then Theorem~\ref{t:int}-(ii) would follow directly from the $h$-invariance of $\Lambda$. But, in general, the leaves of $(\FF\times I)|_{N_0}$ may not be compact (unless $\FF$ has compact leaves), and thus this more familiar argument cannot be used.
\end{rem}

By Corollary~\ref{c:integrable homotopy to pitchfork ls} and Theorem~\ref{t:int}, the following definition can be given.

\begin{defn}\label{d:Lambda-coincidence}
With the above notation, let $\phi,\psi\in C^{1,0}(\FF,\FF)$ inducing the same map $M/\FF\to M/\FF$. The {\em $\Lambda$-coincidence\/} of $\phi$ and $\psi$ is the number
$$
\Coin_\Lambda(\phi,\psi)=\int_{\Coin_0(\xi,\zeta)}\epsilon_{\xi,\zeta}\,\Lambda
$$
for any $(\xi,\zeta)$ in $\pitchfork^1_{\text{\rm ls}}(\FF,\FF\times\FF;\D_M)$ such that there is an integrable homotopy between $(\phi,\psi)$ and $(\xi,\zeta)$.
\end{defn}

\begin{rem}
Suppose that $\FF'$ is any $C^r$ foliation of the same dimension as $\FF$ on a manifold $M'$.
If $\FF$ and $\FF'$ are oriented, then the definition of the $\Lambda$-coincidence can be similarly defined for foliation maps $\phi,\psi\in C^{1,0}(\FF,\FF')$ inducing the same map $M/\FF\to M'/\FF'$. In this case, when $(\phi,\psi)$ are in $\pitchfork^r_{\text{\rm ls}}\!\!(\FF,\FF'\times\FF';\D_{M'})$, the function $\epsilon_{\phi,\psi}$ is given by the compatibility of $\phi_*-\psi_*:T_x\FF\to T_{\phi(x)}\FF'$ with the given orientations for each $x\in\Coin_0(\phi,\psi)$.
\end{rem}

The following is an elementary consequence of the definition of $\Lambda$-coincidence.

\begin{prop}\label{p:invariance of the Lambda-coincidence}
Let $\phi,\psi,\xi,\zeta\in C^{1,0}(\FF,\FF)$ inducing the same map $M/\FF\to M/\FF$. If there is an integrable homotopy between $(\phi,\psi)$ and $(\xi,\zeta)$, then
$$
\Coin_\Lambda(\phi,\psi)=\Coin_\Lambda(\xi,\zeta)\;.
$$
\end{prop}

\begin{cor}\label{c:Lambda-coincidence is locally constant}
The $\Lambda$-coincidence is locally constant on the space
$$
\{(\phi,\psi)\in C^{1,0}(\FF,\FF\times\FF)\ |\ \text{$\phi$ and $\psi$ induce the same map $M/\FF\to M/\FF$}\}
$$
with the strong plaquewise topology.
\end{cor}

\begin{proof}
This follows directly from Propositions~\ref{p:integrable homotopy} and~\ref{p:invariance of the Lambda-coincidence}.
\end{proof}

\begin{defn}\label{d:Lambda-Lefschetz number}
With the above notation, the {\em $\Lambda$-Lefschetz number\/} of any  $\phi\in C^{1,0}_{\text{\rm leaf}}(\FF,\FF)$ is
$$
L_\Lambda(\phi)=\Coin_\Lambda(\Graph(\phi))\;.
$$
\end{defn}

The following are particular cases of Proposition~\ref{p:invariance of the Lambda-coincidence} and Corollary~\ref{c:Lambda-coincidence is locally constant}.

\begin{cor}
Integrably homotopic maps in $C^{1,0}_{\text{\rm leaf}}(\FF,\FF)$
have the same $\Lambda$-Lef\-schetz number.
\end{cor}

\begin{cor}
The $\Lambda$-Lefschetz number is locally constant on $C^{1,0}_{\text{\rm leaf},SP}(\FF,\FF)$.
\end{cor}

For any $\phi\in\Diff^1(\FF)$, by Corollary~\ref{c:Diff r(cF) cap Graph -1(pitchfork r ls(cF,cF times cF;D M))}, there is some $\psi\in\Diff^1(\FF)$ such that there is an integrable homotopy between $\phi$ and $\psi$, and $\Graph(\psi)\pitchfork_{\text{\rm ls}}\D_M$. Then
$$
L_\Lambda(\phi)=\int_{\Fix_0(\phi)}\epsilon_\phi\,\Lambda\;.
$$

Let $\phi:\FF\to\FF$ be a foliation homeomorphism, which possibly does not preserve each leaf. For any transversal $T$ of $\FF$, the image $\phi(T)$ is another transversal. Thus we can consider the restriction of $\Lambda$ to $T$ and $\phi(T)$. It is said that $\Lambda$ is {\em invariant\/} by $\phi$ if the homeomorphism $\phi:T\to\phi(T)$ preserves $\Lambda$ for any transversal $T$.

\begin{prop}
If $\phi,\psi\in\Diff^1(M,\FF)$ induce the same map $M/\FF\to M/\FF$ and $\Lambda$ is $\phi$-invariant, then
$$
\Coin_\Lambda(\phi,\psi)=L_\Lambda(\psi\circ\phi^{-1})\;.
$$
\end{prop}

\begin{proof}
This follows because the composite~\eqref{e:(phi,psi)} equals $(\phi,\psi)$.
\end{proof}

\begin{rem}
It is clear that, instead of assuming that $M$ is closed, we can assume that $\Lambda$ is compactly supported.
\end{rem}


\begin{thebibliography}{10}

\bibitem{AlvKordy:Betti}
J.A. {\'{A}lvarez L\'{o}pez} and Y.A. Kordyukov.
\newblock Distributional {Betti} numbers of transitive foliations of codimension one.
\newblock In {\em Foliations: Geometry and Dynamics (Warsaw, 2000)\/}, ed. P.
Walczak et al. World Scientific, Singapore, 2002, pp.~159--183.

\bibitem{AlvKordy}
J.A.~\'Alvarez L\'opez and Y.A.~Kordyukov. Lefschetz distribution of
Lie foliations. In {\em $C^*$-algebras and elliptic theory II\/}.
Trends in Mathematics, pp. 1 -- 40, Birkh\"auser, Basel, 2008.

\bibitem{AlvMasa2006}
J.A.~\'Alvarez L\'opez and X.~Masa.
Morphisms of pseudogroups and foliated maps.
In {\em Foliations 2005\/}, pp.~1--19, World Sci. Publ., Hackensack, NJ, 2006.

\bibitem{AlvMasa}
J.A.~\'Alvarez L\'opez and X.~Masa.
Morphisms between complete Riemannian pseudogroups.
{\em Topology and its Applications\/}, to appear.

\bibitem{Benameur}
M.T.~Benameur.
A longitudinal Lefschetz theorem in
$K$-theory. {\em $K$-Theory\/}  12  (1997), no. 3, 227--257.

\bibitem{Bermudez}
M.~Berm\'udez. Sur la caract\'eristique d'Euler des feuilletages
mesur\'es. {\em J. Funct. Anal.\/} 237 (2006), no.~1, 150--175.

\bibitem{CandelConlon}
A.~Candel and L.~Conlon. {\em Foliations. I\/}.
Graduate Studies in Mathematics, 23. Amer. Math. Soc.,
Providence, RI, 2000.

\bibitem{Connes1979}
A.~Connes.
\newblock Sur la th\'eorie non commutative de l'int\'egration.
\newblock In {\em Alg\`ebres d'op\'erateurs (S\'em., Les Plans-sur-Bex, 1978)},
  {Lecture Notes in Math.} Vol.~725, pp.~19--143. Springer, Berlin, Heidelberg, New York,
  1979.

\bibitem{Connes1982}
A.~Connes.
A survey of foliations and operator algebras. In {\em Operator algebras and applications, Part I (Kingston, Ont., 1980)\/}, pp. 521--628,
Proc. Sympos. Pure Math., 38, Amer. Math. Soc., Providence, R.I., 1982.


\bibitem{ConnesSkandalis}
A.~Connes and G.~Skandalis.
The longitudinal index theorem for foliations.
{\em Publ. Res. Inst. Math. Sci.\/}  20  (1984), no. 6, 1139--1183.

\bibitem{Haefliger1980}
A.~Haefliger. Some remarks on foliations with minimal leaves.
{\em J. Diff. Geom.\/}  15  (1980), no.~2, 269--284.

\bibitem{HeitschLazarov}
J.L.~Heitsch and C.~Lazarov.
A Lefschetz theorem for foliated manifolds.
{\em Topology\/} 29 (1990), 127--162.

\bibitem{Hirsch}
M.W.~Hirsch. {\em Differential Topology\/}. Corrected reprint of the
1976 original. Graduate Texts in Mathematics, 33. Springer-Verlag,
New York, 1994.

\bibitem{Kechris}
A.S.~Kechris.  {\em Classical Descriptive Set Theory\/}. Graduate
Texts in Mathematics, 156. Springer-Verlag, New York, 1995.

\bibitem{Mumken}
B.~M\" umken. On tangential cohomology of {Riemannian} foliations.
{\em Amer. J. Math.\/} 128 (2006), 1391--1408.
\end{thebibliography}
\end{document}